\documentclass[a4paper,10pt]{article}

\usepackage{amsmath}

\usepackage{amsthm}
\usepackage[dvips]{graphicx}

\usepackage{epsfig}

\usepackage{psfrag}

\usepackage[all]{xy}

\usepackage{amssymb}
\usepackage{latexsym}

\newtheorem{Theorem}{Theorem}[section]

\newtheorem{Definition}[Theorem]{Definition}
\newtheorem{Proposition}[Theorem]{Proposition}
\newtheorem{Lemma}[Theorem]{Lemma}
\newtheorem{Corollary}[Theorem]{Corollary}
\newtheorem{Remark}[Theorem]{Remark}
\newtheorem{Claim}[Theorem]{Claim}
\newtheorem{Assumption}[Theorem]{Assumption}

\def\Axis{\operatorname{A}}
\def\dist{\operatorname{dist}}
\def\Fix{\operatorname{A_{\infty}}}

\begin{document}

\title{Regenerating hyperbolic cone 3-manifolds from dimension 2}

\author{Joan Porti\footnote{Partially supported by the 
Spanish Micinn through grant MTM2009-07594 and prize ICREA ACADEMIA 2008}} 
\date{\today}

\maketitle

\begin{abstract}
We prove that a closed 3-orbifold that fibers over a hyperbolic polygonal
2-orbifold admits a family
of hyperbolic cone structures that are viewed as regeneration of the polygon,
provided that the perimeter is minimal.
\end{abstract}

\section{Introduction}
\label{section:introduction}

The   space of hyperbolic cone 3-manifolds with 
fixed topological type and with cone angles less than
$\pi$ is well understood by \cite{Weiss}, but  the boundary of this space is not. 
The aim of this paper is to establish a
regeneration result in the spirit of Hodgson \cite{Hodgson}, that goes
from a hyperbolic 2-orbifold, viewed as a collapsed 3-orbifold,  to a family of hyperbolic
cone 3-manifolds with decreasing cone angles, starting at $\pi$.
The 3-orbifold is Seifert fibered, and those cone manifolds appear in the
proof of the orbifold theorem;
however, none of the approaches shows the explicit collapse, as 
the Seifert fibration is constructed by other methods 
\cite{CHK,BP,BLP,BLP2}.

Let $\mathcal O^3$ be a closed and orientable 3-orbifold, which is  Seifert
fibered over a two orbifold $P^2$:
$$
	S^1\to \mathcal O^3\to P^2.
$$
The branching locus of $\mathcal O^3$
is a link or a trivalent graph $ \Sigma_{\mathcal O^3}$.
Its edges and circles are grouped in two,
horizontal (if they are transverse to the fibers)
or vertical (if they are fibers):
$$
\Sigma_{\mathcal O^3}=\Sigma_{\mathcal O^3}^{Hor}\cup \Sigma_{\mathcal O^3}^{Vert}.
$$
Points in $\Sigma_{\mathcal O^3}^{Hor}$
project to the mirror and dihedral points of $P^2$. 
Assume
that the orbifold $P^2$ is a \emph{hyperbolic} Coxeter group, generated by reflections on a
hyperbolic polygon whose angles are $\pi$ over an integer. Thus $P^2$ is a
polygon with mirror points at the edges, 
and dihedral points at the vertices.
We may assume also that $P^2$ 
has possibly a single cone point in its interior.
For instance, $S^3$ with branching locus a Montesinos link is an example of such
fibration.

We view the Seifert fibration as a transversely hyperbolic foliation, hence with a developing map 
$$
D_0\! : \! \tilde{ \mathcal O}^3\to \mathbf H^2
$$ 
that factors through the universal covering of $P ^2$.

According to Kerckhoff's proof of Nielsen conjecture \cite{KerckhoffAnnals},
there is a unique point in the Teichm\"uller space that minimizes the perimeter
of $P^2$. Let 
$$
P^2_{min}
$$
denote the orbifold equiped with this hyperbolic structure.

The main result of this paper is the following:

\begin{Theorem}
\label{theorem:1}
Assume that $P^2$ has at most one cone point in its interior. 
 There exists a family of hyperbolic cone manifold structures $C(\alpha)$ on $\vert
\mathcal O^3 \vert$, with singular locus $\Sigma_{\mathcal O^3}$
and cone angle $\alpha\in (\pi-\varepsilon,\pi)$ on $\Sigma_{\mathcal O^3}^{Hor}$ and constant
angles (the orbifold ones) on $\Sigma_{\mathcal O^3}^{Vert}$,
so that
$$
\lim_{\alpha\to \pi^-} C(\alpha)= P^2_{min}
$$
for the Gromov-Hausdorff convergence.
 Moreover the developing maps converge to the developing
map of the transversely hyperbolic foliation.
\end{Theorem}

The orbifold $\mathcal O^3$ can be viewed as a cone manifold $C(\pi)$ with geometry $\widetilde{SL_2(\mathbf R)}$ or $\mathbf H^2\times\mathbf R$, depending on the Euler number. 
With one of these geometries, the singular edges and circles are either horizontal (with angle $\pi$) or vertical, as those geometries are fibered.
The theorem generalizes to those cone manifolds, without assuming that the vertical angles are $2\pi/n$, just that the basis is a polygon 
with angles $\leq \pi/2$. This means that if $q\geq 1$ is the order of a singular fiber (with $q=1$ if the fiber is regular),
and if $\vartheta_i\leq 2\pi$ is the cone angle, then we require $\vartheta_i/q\leq  \pi$.

Let $k$ be the number of circles and edges  on $\Sigma_{C(\pi)}^{Hor}$.
Choose weights
$w_1,\ldots,w_k\in\mathbf R_+=\{x\in\mathbf R\mid x>0\}$ on these  horizontal components. 
Those induce weights 
$$
\mathcal W=\{\overline w_1,\ldots,\overline w_n\}
$$ 
on the edges of $P^2$, just by adding the two weights of 
the mirror points of the fiber of any interior point of the edge,
ie.\ $\overline w_i=w_{j_i}+w_{k_i}$.
Of course if $k=1$, then $\mathcal W$ is the constant weight. 
If $e_1,\ldots,e_n$ are the edges of $P^2$, with lengths $\vert e_1\vert ,\ldots,\vert e_n\vert $,  
the $\mathcal W$-perimeter is defined as
$$
\overline w_1\vert e_1\vert+\cdots+\overline w_n\vert e_n\vert.
$$
We will show (Proposition~\ref{prop:geometry}) that in the set of hyperbolic structures on a polygon
with ordered angles $\vartheta_1/2q_1,\ldots, \vartheta_n/2q_n$ there is a unique minimizer of the $\mathcal W$-perimeter,
that we denote
$
P^2_{\mathcal W-min}.
$

\begin{Theorem}
\label{theorem:2}
Let $C(\pi)$ be a closed, orientable  cone manifold with geometry $\widetilde{SL_2(\mathbf R)}$ or $\mathbf H^2\times\mathbf R$
with cone angles $\leq 2\pi$, and
with space of fibers a polygon $P^2$ with angles $\leq \pi/2$.
There exists a family of hyperbolic cone manifold structures $C(\alpha_1,\ldots,\alpha_k)$
with cone angles $\alpha_i=\pi-w_i t$, for $i=1,\ldots,k$ and   $t\in (0,\varepsilon)$ on $\Sigma_{\mathcal O^3}^{Hor}$, and constant
angles  on $\Sigma_{\mathcal O^3}^{Vert}$, such that
$$
\lim_{t\to 0^+} C(\alpha_1,\ldots,\alpha_k)=P^2_{\mathcal W-min}
$$
for the Gromov-Hausdorff converge.
\end{Theorem}

A version of Theorem~\ref{theorem:1}  was stated in Hodgson's thesis \cite{Hodgson}. 
In particular, Hodgson showed that the minimizer of the perimeter corresponds
to a singularity in the variety of representations of
$\mathcal O^3\smallsetminus\Sigma_{\mathcal O^3}$.
However, to construct the path of representations requires further detail, done here.
For the construction of the developing maps we use an approach different from \cite{Hodgson}.

The idea of regenerating from Kerckhoff minimizers is also related to a
theorem of  Series on pleated surfaces \cite{Series},
where  as an application she also obtains a family of cone manifolds.

Theorem~\ref{theorem:1}
 needs to assume that there is at most one cone point in the interior of the
polygon, otherwise  $\mathcal O^3\setminus \Sigma_{\mathcal O^3}$
would contain an essential torus, contradicting  the existence of hyperbolic
cone structures.

\begin{figure}
\begin{center}
{
\includegraphics[height=4cm]{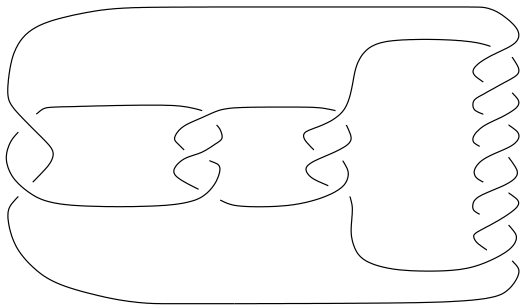}
}
\end{center}
   \caption{Example of fibered orbifold: $S^3$ with branching locus this link.
The base is a quadrilateral.}\label{fig:4pretzel}
\end{figure}

\smallskip

We make   the following assumption along the paper:

\begin{Remark}
\label{remark:nointeriorpoint}
We may assume that  $P^2$ has no interior cone point:
If the interior of $P^2$ has one cone point, then we consider the orbifold covering
that unfolds this point and work equivariantly.
\end{Remark}

Concerning the equivariance, the structure  of $P^2$ that minimizes the perimeter is unique, by Kerckhoff \cite{KerckhoffAnnals},
hence equivariant. The structure of the hyperbolic cone manifolds is also equivariant because it is 
unique, by Weiss' global rigidity \cite{WeissGlobal}.
Notice that two or more cone points would not produce a polygon when unfolding, but an orbifold with more complicated underlying space.

\smallskip

Let us show a consequence of Theorem~\ref{theorem:2}  for hyperbolic polyhedra.
Fix $n$ positive real numbers 
$$
0<\beta_1,\ldots,\beta_n\leq \pi/2,
$$
satisfying $\sum(\pi-\beta_i)> \pi $. By Andreev theorem, for any choice of $\alpha_1,\ldots,
\alpha_n$, $\alpha_1',\ldots,\alpha_n'$
 satisfying
$$
0< \alpha_i,\alpha_i'<\pi/2,\qquad \alpha_i+\alpha_{i+1}>\pi-\beta_i,\quad \textrm{ and }\quad \alpha_i'+\alpha_{i+1}'>\pi-\beta_i,
$$
there exists a unique hyperbolic polyhedron with the combinatorial type of a
prism with an $n$-edged polygonal base, with dihedral angles at the ``vertical'' edges
 $\beta_1,\cdots,\beta_n$, angles $\alpha_1,\ldots,\alpha_n$ at the respective $n$ ``horizontal'' edges of the top face,
and $\alpha_1',\ldots,\alpha_n'$ at the respective $n$ horizontal edges of the bottom face.
They are arranged so that the edges with angles $\alpha_i$, $\beta_i$, $\alpha_i'$ and $\beta_{i+1}$ bound a quadrilateral
face.
See Figure~\ref{fig:prism}.

\begin{figure}
\begin{center}
{
\psfrag{a1}{$\alpha_1$}
\psfrag{a2}{$\alpha_2$}
\psfrag{a3}{$\alpha_3$}
\psfrag{a4}{$\alpha_4$}
\psfrag{a5}{$\alpha_5$}
\psfrag{a1'}{$\alpha_1'$}
\psfrag{a2'}{$\alpha_2'$}
\psfrag{a3'}{$\alpha_3'$}
\psfrag{a4'}{$\alpha_4'$}
\psfrag{a5'}{$\alpha_5'$}
\psfrag{b1}{$\beta_1$}
\psfrag{b2}{$\beta_2$}
\psfrag{b3}{$\beta_3$}
\psfrag{b4}{$\beta_4$}
\psfrag{b5}{$\beta_5$}
\includegraphics[height=4cm]{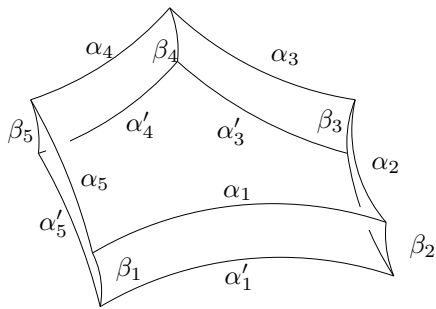}
}
\end{center}
   \caption{A hyperbolic prism as in Corollary~\ref{prop:prism}.}\label{fig:prism}
\end{figure}

Now if we choose weights $w_1,\ldots,w_n>0$ for the bottom horizontal edges and
 $w_1',\ldots,w_n'>0$ for the top ones, assume that 
\begin{eqnarray*}
 \alpha_i &= &\pi/2-w_i t, \\
 \alpha_i'& = &\pi/2-w_i' t.
\end{eqnarray*}
Keep $\beta_1,\ldots,\beta_n$ fixed and let $t\searrow 0$ (hence $\alpha_i,\alpha_i'\nearrow\pi/2$).
Set 
$$
\mathcal W=\{w_1+w_1',\ldots,w_n+w_n'\}
$$ and recall that the $\mathcal W$-perimeter 
of a polygon is the addition of its edge lengths multiplied by the weights.

\begin{Corollary}
 \label{prop:prism}
When $t\searrow 0$, the prism converges to the $n$-edged polygon
with angles $\beta_1,\ldots,\beta_n$ of minimal $\mathcal W$-perimeter, where
$\mathcal W=\{w_1+w_1',\ldots,w_n+w_n'\}$.
\end{Corollary}

The following result follows from the proof of Theorem~\ref{theorem:2} 
and the results of Appendix~\ref{sec:appendix}. 
Given $0<\vartheta_1,\ldots,\vartheta_n\leq  \pi/2$ with $\sum (\pi-\vartheta_i)>2\pi$, 
let $$
\mathfrak P(\vartheta_1,\ldots,\vartheta_n)
$$ 
denote the space of hyperbolic n-gons with those ordered angles.

\begin{Proposition}
\label{prop:geometry}
The perimeter has a unique minimum in $\mathfrak P(\vartheta_1,\ldots,\vartheta_n)$. In addition, this is the only 
 polygon in $\mathfrak P(\vartheta_1,\ldots,\vartheta_n)$
 that has an inscribed circle tangent to all 
of its edges.

For $\mathcal W=\{w_1,\ldots,w_n\}$, with $w_i>0$, the $ \mathcal W$-perimeter has a unique minimum in $\mathfrak P(\vartheta_1,\ldots,\vartheta_n)$.
In addition, this is the only polygon in $\mathfrak P(\vartheta_1,\ldots,\vartheta_n)$ 
such that has a point $p$ in its interior, so that 
$\frac1{w_i}\sinh(d(e_i,p))$ is independent of  the edges $e_i$ of the polygon.
\end{Proposition}

To prove Proposition~\ref{prop:geometry}, in Appendix~\ref{sec:appendix} we adapt to this setting the remarkable work of Kerckhoff \cite{KerckhoffAnnals}, Thurston \cite{ThurstonEarthquake}
and Wolpert \cite{Wolpert}
on earthquakes and convexity of length functions, and a generalization of Thurston's earthquake theorem
to cone manifolds due to Bonsante and Schlenker \cite{BonsanteSchlenker}.
In addition, in the proof of Theorem~\ref{theorem:1} there is Killing vector 
field (corresponding to the deformation of the fiber) 
that has an axis perpendicular to the polygon and is equidistant to  all of its edges. In particular
the polygon that minimizes the perimeter has an inscribed circle tangent to all 
 of its edges. Uniqueness of the polygon in $\mathfrak P(\vartheta_1,\ldots,\vartheta_n)$
with such an inscribed circle
is easy to prove by a continuity method, however I was unable to 
find Proposition~\ref{prop:geometry} in the literature.


\medskip

The proof of Theorem~\ref{theorem:1} has two parts: first to construct a curve of representations of the smooth
part of $\mathcal O^3$ and second to prove that these representations 
are holonomy structures of the cone manifolds by constructing developing maps.

For the construction of the curve, we have to choose the structure that  minimizes 
the
perimeter. Using the symplectic structure of the variety of representations
of $\partial \mathcal N(\Sigma_{\mathcal O}) $, due to Goldman \cite{GoldmanAdvances},
the Hamiltonian vector field of the perimeter is essentially the direction to 
regenerate in the variety of representations of $\partial \mathcal N(\Sigma_{\mathcal O}) $. 
Beeing a 
critical point for the perimeter implies that this direction is induced from 
deformations of $\mathcal O\setminus \mathcal N(\Sigma_{\mathcal O})$.

The construction of the curve of representation is much easier if the base $P^2$
is a triangle, ie.\ the orbifold is small. In this case the Teichm\"uller
space is just a point, and a direct computation in cohomology  allows to construct 
the curve of representations. Notice also that, since the orbifold is small, the
analysis of Paiva-Barreto \cite{Alexandre} applies here, to prove that the limit of cone manifolds is 
the 2-dimensional orbifold.

Once we have the existence of the curve of representations, we 
 construct developing maps. For this we use 
the fibration: vertices of the base correspond to
rational  tangles, edges to $I$-fibered strips, and the interior points to
regular fibers. We
construct the developing map  first for the tangles, then for the strips that
connect them, and finally for the regular points. 
In particular the union of tangles and neighborhoods of the strips 
is a solid torus, and the underlying space of the orbifold is a generalized lens
space. Previous to this construction, we must analyze the infinitesimal deformation of
the fiber, and the corresponding Killing vector field, which happens to be perpendicular
to the developing map of the two dimensional polygon.

\smallskip

The proof of Theorem~\ref{theorem:2} follows exactly the same scheme as Theorem~\ref{theorem:1},
just by adding the weights, and by adapting some arguments from orbifolds to cone manifolds.
To simplify, we discuss first  Theorem~\ref{theorem:1} and  prove
Theorem~\ref{theorem:2} in 
Section~\ref{section:weights}.

\medskip

The paper is organized as follows. 
In Section~\ref{sec:varieties of representations} we state the existence of a one parameter
deformation of representations with some properties. 
We prove it in this section  when $P^2$ is a 
triangle, and  its proof for general $P^2$ is done in the following sections.
Section~\ref{section:relative} deals with relative character varieties.
Section~\ref{section:fenchelnielsen} is devoted 
the symplectic structure of the variety of characters of a surface, and to
 Fenchel-Nielsen local coordinates.
The curve of representations is build in Section~\ref{section:adding}
when all singular fibers are in $\Sigma_{\mathcal O}^{Vert}$,
and 
Section~\ref{section:curve}
in the general case.
In Section~\ref{section:fibration} we recall
the structure of the fibration of $\mathcal O^3$ and we also set notation.
Section~\ref{section:Killing} is devoted to the the study of the Killing
vector field associated to the infinitesimal deformation of the fiber.
Developing maps are  constructed in Section~\ref{section:developingmaps}.
In Section~\ref{section:weights}, we discuss  
Theorem~\ref{theorem:2}. 
Section~\ref{section:example} is devoted to an example.
Finally, Appendix~\ref{sec:appendix} explains earthquakes and Kerckhoff-Thurston-Wolpert
theory for cone manifolds, in particular Proposition~\ref{prop:geometry} is proved in this appendix.
 Appendix~\ref{sec:inf} is devoted to some results about
infinitesimal isometries, used mainly in Section~\ref{section:Killing}.


\section{Varieties of representations}
\label{sec:varieties of representations}

We start with the holonomy representation of the hyperbolic orbifold 
$$
\operatorname{hol}\!:\! \pi_1(P^2)\to PGL_2(\mathbf R)=\operatorname{Isom}(\mathbf H^2),
$$
where the elements preserve or reverse the orientation of $\mathbf H^2$ according to the sign of
the determinant. 
Notice that 
$$
	PGL_2(\mathbf R)=PSL_2(\mathbf R)\sqcup PSL_2(\mathbf i\mathbf R) < PSL_2(\mathbf C)
	=\operatorname{Isom}^+(\mathbf H^3).
$$
Let $$
M=\vert \mathcal O^3\vert\setminus \mathcal N(\Sigma_{\mathcal O^3})
$$ 
denote the smooth
part of the orbifold.
By \cite{culler,FicoMontesinos}, the induced representation on $M$ can be lift
to 
$$
\rho_0:\pi_1(M)\to SL_2(\mathbf C).
$$

The goal of this section and the next ones is to prove the following result.


\begin{Proposition}
\label{assumption:rho} There exists 
$\{\rho_t\}_{t \in (-\varepsilon,\varepsilon)}$  an analytic path  of
representations of $M$ in $SL_2(\mathbf C)$
such that $\rho_0$ is as above and for each $t\in (0,\varepsilon)$, $\rho_t$ of a vertical meridian is constant, and 
$\rho_t$ of a horizontal
meridian
is a rotation of
angle  
$$
\pi-t^r+ O(t^{r+1})
$$ 
for some $r\in\mathbf Z$, $r>0$,
independent of $i$.
\end{Proposition}

It is convenient to fix the orientation of the singular edges and their meridians to distinguish a 
rotation of angle $\pi-t^{r}$ from $\pi+t^{r}$, $t>0$.

We prove that this  proposition  holds when $P^2$ is a triangle at the end of this section,
but the general case is proved in Sections~\ref{section:adding}  and \ref{section:curve}. 
Later, we will also show that $r=1$.
Before  that,
we state some basic properties of the variety of representations
and characters.

\medskip

For an orbifold or a manifold $Z$, the \emph{variety of representations}
of $\pi_1(Z)$ in $SL_2(\mathbf C)$ is 
$$
R( Z)=\hom(\pi_1(Z),SL_2(\mathbf C)). 
$$
This is a complex affine set of $\mathbf C^N$ defined over $\mathbf Q$.
The embedding in $\mathbf C^N$ is given by trace functions of $N$ elements of $\pi_1( Z)$
\cite{CullerShalen,FicoMontesinos}.

For a representation $\rho\in R(Z)$, its character is the map
$$
\begin{array}{rcl}
 \chi_{\rho}:\pi_1(Z)&\to&\mathbf C\\
\gamma & \mapsto & \operatorname{Trace}(\rho(\gamma)).
\end{array}
$$
The \emph{variety of characters} $X(Z)$ is the set of all characters of $R(Z)$,
and it is also a complex affine set over $\mathbf Q$.

A representation $\rho\in R(Z)$ is called \emph{irreducible} if no proper subspace of
$\mathbf C^2$ is $\rho(\pi_1(Z))$-invariant. The representations we are considering are always irreducible.
The set of irreducible representations is Zariski open, and so is the set of irreducible characters
\cite{CullerShalen}. We denote them by
$R^{irr}(Z)$ and $X^{irr}(Z)$ respectively.

\begin{Lemma}[\cite{CullerShalen}]
\label{lemma:varietycharacters} 
The projection
$$
\begin{array}{rcl}
 R(Z) & \to & X(Z) \\
 \rho & \mapsto & \chi_{\rho}
\end{array}
$$
is surjective.
Moreover  $R^{irr}(Z)\to X^{irr}(Z)$ is a local fibration with fiber the orbit by conjugation.
\end{Lemma}

In fact, the action of $SL_2(\mathbf C)$ on $R(Z)$ by conjugation is algebraic, and $X(Z)$ is the Mumford quotient
in algebraic invariant theory.

Since orbifolds have torsion, sometimes we  need to work with representations in $PSL_2(\mathbf C)$,
because they may not lift to $SL_2(\mathbf C)$. This does not make any difference
for the local structure of the variety of representations and characters at the representations 
we are interested in, cf.\ \cite{HP2}. The varieties of $PSL_2(\mathbf C)$ representations and characters are denoted by
$$
R_{PSL_2(\mathbf C)}(Z) \quad \textrm{ and }\quad X_{PSL_2(\mathbf C)}(Z).
$$

By Weil's construction, the Zariski tangent space to $X(Z)$ at $\chi_{\rho}$ is naturally identified with the 
first cohomology group of $\pi_1( P^2)$ with coefficients in the Lie algebra $\mathfrak{sl}_2(\mathbf C)$
twisted by the adjoint representation $Ad\rho$ \cite{LubotzkiMagid}:

\begin{Lemma}[\cite{Weil, LubotzkiMagid,HP2}]
\label{lemma:weil}
If $\rho \in R(Z)$ is irreducible, then  
$$
T^{Zar}_{\chi_{\rho}}X(Z)\cong H^1(\pi_1(Z),Ad\rho),
$$
where $T^{Zar}$ means the Zariski tangent space as a scheme
(not necessarily reduced).

If $\rho \in R_{PSL_2(\mathbf C)}(Z)$ does not preserve a subset of  $\partial_{\infty} \mathbf H^3$ of cardinality $\leq 2$, then
$$
T^{Zar}_{\chi_{\rho}}X_{PSL_2(\mathbf C)}(Z)\cong H^1(\pi_1(Z),Ad\rho).
$$
\end{Lemma}

Notice that the hypothesis that $\rho$ does not preserve a  subset of  $\partial_{\infty} \mathbf H^3$ of cardinality $\leq 2$,
is equivalent to say 
that $\rho$ is irreducible and that it does not preserve any unoriented geodesic.

We may need to work with  cohomology of orbifolds instead of manifolds: a possible way to define it
is as the equivariant 
cohomology of a manifold cover (all orbifolds here are very good: they have a finite covering which is a manifold).
Otherwise, it can be equivalently defined  as simplicial cohomology
of a triangulation adapted to the stratification of the singularity.

\begin{Lemma}
\label{lemma:hohomology} 
For a very good orbifold $Z$ there is a natural map
$$
H^i(\pi_1(Z); Ad\rho)\to H^i(Z; Ad\rho)
$$
which is  an isomorphism for $i\leq 1$ and injective for $i=2$.
\end{Lemma}

\begin{proof}
When $Z$ is a manifold, there is always a natural map from $Z$ to a $K(\pi,1)$ space, that consist in attaching cells to $Z$ 
to kill the higher homotopy groups. Since the attached cells are of dimension $\geq 3$, the lemma follows for manifolds
(cf. \cite[Lemma~3.1]{HPAGT}).
 Since the coefficients are 
$\mathfrak{sl}_2(\mathbf C)$, by working equivariantly it also holds true for very good orbifolds.
\end{proof}

See \cite{CullerShalen,FicoMontesinos,GoldmanAdvances,HP2, LubotzkiMagid}
for more results about the varieties of representations and characters.


\begin{proof}[Proof of Proposition~\ref{assumption:rho} when $P^2$ is a triangle]
Let $\rho_0\!:\!\pi_1(P^2)\to PSL_2(\mathbf C)$ denote the holonomy representation. (It also denotes the lift to
$SL_2(\mathbf C)$ of the induced representation on $M=\mathcal O^3\setminus\Sigma_{\mathcal O^3}$.)
As $P^2$ is an orbifold, we have to work with coefficients in $PSL_2(\mathbf C)$ instead of
$SL_2(\mathbf C)$.
 Triangular orbifolds 
are rigid, and therefore the tangent space to the variety of $PSL_2(\mathbf C)$-characters of $\pi_1 (P^2)$ is trivial at 
$\rho_0$:
$$
H^1(\pi_1(P^2); Ad\rho_0)=
H^1(P^2; Ad\rho_0)=0.
$$

 The variety of $PSL_2(\mathbf C)$-characters for 
$\mathcal O^3$ is locally isomorphic to the one of $P^2$, because any irreducible 
representation of $\mathcal O^3$ must map the fiber to the center, hence to the 
identity matrix, cf.\ Lemma~\ref{Lemma:basefibr}. It follows that
$$
H^1(\mathcal O^3; Ad\rho_0)=0.
$$
Also $H^0(\mathcal O^3; Ad\rho_0)=0$ because the representation is irreducible, thus
by duality:
$$
H^*(\mathcal O^3; Ad\rho_0)=0.
$$
This vanishing does not hold true if $P^2$ is a large orbifold.

 By  the Mayer-Vietoris exact sequence of the pair $(M,\mathcal N(\Sigma_{\mathcal O^3}))$, 
we get the isomorphism: 
$$
0\to H^1(M; Ad\rho_0)\oplus H^1( \mathcal N(\Sigma_{\mathcal O^3}); Ad\rho_0)\to H^1(\partial M; Ad\rho_0)\to 0.
$$
Its dual in homology is
\begin{equation}
\label{eqn:MVhom}
0\to H_1(\partial M; Ad\rho_0)\to  H_1(M; Ad\rho_0)\oplus H_1( \mathcal N(\Sigma_{\mathcal O^3}); Ad\rho_0)\to 0.
 \end{equation}
In particular, $\dim H_1(M; Ad\rho_0)=\frac12\dim H_1(\partial M; Ad\rho_0)  $.

Let $\gamma_1,\ldots , \gamma_k\in\pi_1(M)$ denote the meridians of $\Sigma_{\mathcal O^3}$, if 
 $\Sigma_{\mathcal O^3}$ has $ k$ components and edges.
Let  $ \mu_1,\ldots,\mu_k$ denote the complex lengths the meridians   $\gamma_1,\ldots , \gamma_k$
(ie.\ the eigenvalues of a representation evaluated at $\gamma_j$ are  $\pm e^{\pm \mu_j/2}$).

Let $ \lambda_1,\ldots,\lambda_k$ be the complex length of the 
twist parameters of this pants decomposition,
cf.\ Section~\ref{section:fenchelnielsen}. By 
\cite{HPAGT} and Proposition~\ref{prop:FenchelNielsen}, 
$$
\{d{\mu_1},d{\lambda_1},\ldots d{\mu_k},d{\lambda_k}\}
$$
is a basis for the cotangent space of the product of character varieties of the boundary components:
$$
X(\partial M)= X(\partial_1 M)\times\cdots\times  X(\partial_r M),
$$
which is isomorphic to
$$
H_1(\partial M; Ad\rho_0)\cong H_1(\partial_1 M; Ad\rho_0)\oplus\cdots\oplus H_1(\partial_r M; Ad\rho_0).
$$
On the other hand, $
d{\mu_1},\ldots, d{\mu_k}
$
are mapped to zero in 
$H_1( \mathcal N(\Sigma_{\mathcal O^3}); Ad\rho_0)$.
Thus it follows from Isomorphism (\ref{eqn:MVhom}) above that 
$\{d{\mu_1},\ldots d{\mu_k}\}$ is a basis for
the Zariski cotangent space $H_1(M; Ad\rho_0)$.

Now we want to prove that all infinitesimal deformations of $\rho_0\vert_{\pi_1M}$ are integrable: namely that all elements
of the Zariski tangent space $H^1(M;Ad\rho_0)$ are actually tangent vectors to paths. There is
an infinite sequence of obstructions to integrability living 
in the second cohomology group \cite{GoldmanAdvances}, starting with the cup product
and following with Massey products. These obstructions are natural and 
they vanish 
for $\partial M$.
In addition, we also have the following isomorphism from Mayer-Vietoris:
$$
0\to H^2(M; Ad\rho_0) \to H^2(\partial M; Ad\rho_0)\to 0.
$$
Hence the infinite sequence of obstructions to integrability vanishes. 
This only  implies that the infinitesimal deformations are   formally integrable,
but a 
 theorem of Artin implies that 
they are actually integrable \cite{Artin}. See \cite{HPS} for details. Thus
$(\mu_1,\ldots,\mu_k)$ define local coordinates for the variety of
characters of $M$. To prove Proposition~\ref{assumption:rho}, it suffices to 
take $\mu_j= \mathbf i (\pi-t)$ when $\gamma_j$ is horizontal,
and $\mu_j=ctnt$ when $\gamma_j$ is vertical.
\end{proof}

\section{Relative character variety}
\label{section:relative}

Let $Z$ be a compact aspherical 3-manifold with boundary, for instance 
 the exterior of the singular locus $M=\mathcal O\setminus \mathcal N(\Sigma_{\mathcal O})$.

One way to work with manifolds instead of orbifolds is to use relative character varieties of manifolds. This
is convenient for working also with cone manifolds.

\begin{Definition} Let $ \Gamma=\{\gamma_1,\ldots,\gamma_k\}\subset \pi_1(Z)$ be a finite subset. The \emph{relative character variety} 
with respect  to the values $a_1,\ldots,a_k\in\mathbf C	\setminus\{\pm 2\}$ is
$$
X(Z,\Gamma)= \{ \chi\in X(Z)\mid \chi(\gamma_i)=a_i\textrm{  for } \gamma_i\in \Gamma\}.
$$
\end{Definition}

The r\^ole of the $a_1,\ldots,a_k\in\mathbf C	\setminus\{\pm 2\}$ is not important, and they are not included in the notation. Usually, these
values are clear from the context.

\begin{Lemma}
\label{lemma:tgrelative}
Let $\chi=\chi_\rho\in X(Z)$ be an irreducible character such that $\chi(\gamma)\neq \pm 2$ for $\gamma\in \Gamma$.
\begin{enumerate}
 \item The Zariski tangent space to $X(Z, \Gamma)$ is:
\begin{multline*}
  T^{Zar}_{\chi}X(Z, \Gamma) \cong   \ker (H^1(Z;Ad\rho)\to \oplus_{\gamma\in\Gamma} H^1(\gamma,Ad\rho)) \\
				\cong  \operatorname{Im} (H^1(Z,\Gamma;Ad\rho)\to H^1(Z;Ad\rho)).
 \end{multline*}
 \item The Zariski cotangent space to $X(Z, \Gamma)$ is:
\begin{multline*}
  (T^{Zar}_{\chi})^*X(Z, \Gamma) \cong   H_1(Z;Ad\rho)/ \operatorname{Im}(   \oplus_{\gamma\in\Gamma} H_1(\gamma,Ad\rho)) \to H_1(Z;Ad\rho) ) \\
					\cong  \operatorname{Im}( H_1(Z;Ad\rho)\to H_1(Z,\Gamma;Ad\rho)).
 \end{multline*}
\end{enumerate}
 \end{Lemma}
 
The lemma follows from Lemma~\ref{lemma:weil} by analyzing tangent and cotangent induced maps of the morphism induced by inclusion,
$$
X(Z)\to X(\gamma_1)\times \cdots \times X(\gamma_k),
$$
and from the long exact sequence in cohomology of the pair $(Z,\Gamma)$.

Now assume that  $ \Gamma=\{\gamma_1,\ldots,\gamma_k\}$ is a pants decomposition of $\partial Z$. This is, each component of 
$\partial Z\setminus \Gamma$ is either a pair of pants or a cylinder, and the cardinality of $\partial Z$ is minimal. In particular
$$
k=-\frac32 \chi(\partial Z)+\textrm{ number of components of }\partial Z\textrm{ that are tori}.
$$
Also assume that for some $1\leq l\leq k$, 
$$
\chi(\gamma_1)= \pm 2\cos(\frac{\pi}{n_1}), \ldots, \chi(\gamma_l)=\pm 2\cos(\frac{\pi}{n_l}).
$$
Set $\mathbf n= (n_1,\ldots,n_l)\in\mathbf N^l$.
Consider the orbifold ${Z_{\mathbf n}}$ obtained by adding  $l$  1-handles  branched along the cores of respective orders 
$n_1,\ldots,n_l$,
along the  meridians
$\gamma_1,\ldots ,\gamma_l$:
$$
{Z_{\mathbf n}}= Z\cup D^2(n_1)\times [0,1] \cup\cdots\cup D^2(n_l)\times [0,1].
$$
Here $D^2(n_i)$ denotes the disk quotiented out by a group of rotations of order $n_i$, thus 
$D^2(n_i)\times [0,1]$ is a $1$-handle with branching locus its core.

Fill the spherical components of $\partial{Z_{\mathbf n}}$: ie.\ for every $2$-orbifold
in $\partial{Z_{\mathbf n}}$ that is spherical (isomorphic to $S^2/\Gamma$) attach a ball $B^3/\Gamma$:
$$
{\overline Z_{\mathbf n}}=  {Z_{\mathbf n}} \cup B^3/\Gamma_1\cup \cdots \cup B^3/\Gamma_s.
$$
Those spherical 2-orbifolds either come from pairs of pants that are bounded by three curves with an attached
meridian disk each, or they come from cylinders bounded by a curve with an attached meridian. Notice that  $\pi_1({\overline Z_{\mathbf n}})\cong \pi_1({Z_{\mathbf n}})$.


\begin{Lemma}
\label{lemma:isoorbifoldrel}
The inclusion $Z\to {\overline Z_{\mathbf n}}$ induces an isomorphism of the relative varieties of irreducible characters
$$
  X_{PSL_2(\mathbf C)}^{irr}({\overline Z_{\mathbf n}},  \{\gamma_{l+1},\ldots,\gamma_k\})
\overset{\cong}\longrightarrow
 X_{PSL_2(\mathbf C)}^{irr}(Z,\{\gamma_1,\ldots,\gamma_k\}). 
$$ 
\end{Lemma}

\begin{proof}
We prove that this map is a bijection and also that it induces an isomorphism of Zariski tangent spaces.
The natural surjection $\pi_1(Z)\to \pi_1({\overline Z_{\mathbf n}})$ induces the map of relative $PSL_2(\mathbf C)$-character
varieties that is a bijection, because 
$$
  \pi_1({\overline Z_{\mathbf n}})\cong\pi_1({Z_{\mathbf n}})\cong \pi_1(Z)/\langle {\gamma_1}^{n_1},\ldots, {\gamma_l}^{n_l}\rangle,
$$
and the fact a matrix of $PSL_2(\mathbf C)$ has order $n_i$ is determined by its trace.

 To prove that it is a  local isomorphism at the infinitesimal level, 
we claim that $H^1( {\overline Z_{\mathbf n}}; Ad\rho) \to H^1( Z; Ad\rho)$ is an inclusion, and that the image equals the 
kernel of $ H^1( Z; Ad\rho)\to \bigoplus_{i=1}^l H^1( \gamma_i; Ad\rho) $. 
Since the inclusion ${Z_{\mathbf n}}\subset{\overline Z_{\mathbf n}}$ induces an isomorphism of fundamental groups,
 by Lemma~\ref{lemma:hohomology}
$H^1( {\overline Z_{\mathbf n}}; Ad\rho)\cong 
H^1( {Z_{\mathbf n}}; Ad\rho)$ and we may replace ${\overline Z_{\mathbf n}}$ by ${Z_{\mathbf n}}$ in the claim.

Apply the long exact sequence to the pair $( {Z_{\mathbf n}}, Z)$:
$$
 H^1({Z_{\mathbf n}},Z; Ad\rho)\to H^1( {Z_{\mathbf n}}; Ad\rho) \to H^1( Z; Ad\rho)\to  H^2( {Z_{\mathbf n}},Z; Ad\rho).
$$
Let $V$ be the union of singular 1-handles attached along $\gamma_1,\ldots,\gamma_l$. In particular $V\cup Z={Z_{\mathbf n}}$ and
$V\cap Z\simeq \gamma_1\cup\cdots \cup \gamma_l$. By excision
$$
 H^*({Z_{\mathbf n}},Z; Ad\rho)\cong  H^*(V, V\cap Z; Ad\rho)\cong  \bigoplus_{i=1}^l H^*(D^2(n_i),\gamma_i; Ad\rho).
$$
From the exact sequence of the pair, and by using that $H^1(D^2(n_i); Ad\rho)=0$ because $\pi_1(D^2(n_i))$ is finite
and $H^0( D^2(n_i); Ad\rho) \cong H^0( \gamma_i; Ad\rho)$, we deduce that
$$
H^1(D^2(n_i),\gamma_i; Ad\rho)=0 \textrm{ and }  H^2(D^2(n_i),\gamma_i; Ad\rho) \cong H^1(\gamma_i; Ad\rho).
$$ 
This proves the claim and the lemma.
\end{proof}

Assume that ${\overline Z}$ is an orientable Seifert fibered orbifold with base $B$ and that $\chi$ is an irreducible character. 
Assume also that $\gamma_{l+1},\ldots,\gamma_{k}$ project to peripheral elements of $B$.

\begin{Lemma}
\label{Lemma:basefibr}
The projection
$\overline Z\to B$
induces a map
$$
X_{PSL_2(\mathbf C)}^{irr}(B,\{ \gamma_{l+1},\ldots,\gamma_{k}\}) \to 
X_{PSL_2(\mathbf C)}^{irr}(\overline Z,\{ \gamma_{l+1},\ldots,\gamma_{k}\}) 
$$
that is an isomorphism.
\end{Lemma}

\begin{proof}
As in the previous lemma, we prove that this map is a bijection that induces an isomorphism on Zariski tangent spaces.
The kernel $\ker(\pi_1(\overline Z)\to \pi_1(B))$ is  the center of
 a finite index subgroup $G_0 \vartriangleleft\pi_1(\overline Z)$ (of index at most $2$). 
Since $PSL_2(\mathbf C)$ has no center, every irreducible representation
of $\pi_1(\overline Z)$ maps the center to the identity. Hence $\ker(\pi_1(\overline Z)\to \pi_1(B))$
is mapped to the identity or to a normal subgroup of order two. Again by irreducibility,
$\ker(\pi_1(\overline Z)\to \pi_1(B))$ is mapped to the identity, hence irreducible representations
of $\pi_1(\overline Z)$ are canonically in bijection with representations of $\pi_1(B)$.

To prove the isomorphism in cohomology, we use group cohomology, in particular group cocycles.
We claim that every group cocycle $\theta:\pi_1(\overline Z)\to \mathfrak{sl}_2(\mathbf C)$
maps the fiber to zero and therefore factors through $\pi_1(B)$.
Here a cocycle means that 
$\theta(\gamma_1\gamma_2)=\theta(\gamma_1)+Ad_{\rho(\gamma_1)}\theta(\gamma_2)$, $\forall \gamma_1,\gamma_2\in \pi_1(\overline Z)$.

To prove the claim, if $\gamma_0\in \ker(\pi_1(\overline Z)\to \pi_1(B))$, and $\theta$ is a cocycle,
then for any $\gamma\in G_0<\pi_1(\overline Z)$, the cocycle rule applied to the relation $\gamma_0\gamma=\gamma\gamma_0$ reads
$$
\theta(\gamma_0)+Ad_{\rho(\gamma_0)}\theta(\gamma)= \theta(\gamma)+Ad_{\rho(\gamma)}\theta(\gamma_0).
$$
As $\rho(\gamma_0)=\pm\operatorname{Id}$, it follows that $(Ad_{\rho(\gamma)}-1)\theta(\gamma_0)=0$.
 Since this holds
for every $\gamma\in G_0\vartriangleleft\pi_1(\overline Z)$ and $\rho$ is irreducible, it follows that $\theta(\gamma_0)=0$.
\end{proof}

\section{The symplectic structure of the variety of characters}
\label{section:fenchelnielsen}

As in the previous section,  $Z$ is a compact aspherical 3-manifold with boundary. 
%
Let $X(\partial Z)$ denote the product of character varieties of components of $\partial Z$:
$$
X(\partial Z)=X(\partial_1 Z)\times\cdots\times X(\partial_r Z),
$$
where $\partial Z= \partial_1 Z\cup\cdots\cup \partial_r Z$ is the splitting in connected components.

Choose a \emph{pants} decomposition for the components of  $\partial  Z$.
This is, 
a collection of disjoint simple closed  curves in $\partial Z$,
$$
\gamma_1, \ldots, \gamma_{k},
$$
that cut $\partial Z$ into pairs of pants or cylinders, and the family has minimal cardinality.
Here $k=-\frac32\chi(\partial N)+k_0$, where $k_0$ is the number of components of 
 $\partial Z$ that are tori.

For $j=1,\ldots, k$, 
 let $\mu_j$ denote twice the logarithm of the eigenvalue of the trace of the
$j$-th meridian $\gamma_j$, so that $\mu_j$ has real part the 
translation length and imaginary part the rotation angle of $\rho(\gamma_j)$.
Let $\lambda_j$ denote the twist parameter. 
 Algebraically, when we cut along the (non-separating) 
meridian and write the fundamental group of the surface as an HNN-extension,
$\lambda_j$ is twice the logarithm of the eigenvalue
of the element of the extension. (In the separating case, it is twice
the logarithm of
the eigenvalue of the conjugating factor).
Notice that $\lambda_j$ is only defined after normalization. 

When $\gamma_j$ is the meridian of a cone manifold, by \cite{Weiss},
$\lambda_j$ can be chosen so that its real part is 
the length of the corresponding  singular edge.
For representations of $M=\mathcal O\setminus \mathcal N(\Sigma_{\mathcal O})$,
that factor through $P^2$,  then 
$\sum\lambda_j$ is twice the perimeter of $P^2$.

\begin{Proposition}[Fenchel-Nielsen local coordinates]
\label{prop:FenchelNielsen}
Let $\chi\in X(\partial Z)$ be such that $\chi(\gamma_j)\neq \pm 2$,
and $\chi$ restricted to each pant of $\partial Z\setminus \cup_i\gamma_i$ 
is irreducible.
Then the parameters 
$$
(\mu_1,\ldots,\mu_{k},\lambda_1,\ldots,\lambda_{k})
$$ 
define
local coordinates for $X(\partial Z)$ around $\chi$.
\end{Proposition}

Though it is well
known, we give a proof in this algebraic setting for completeness. 

\begin{proof}[Proof of Proposition~\ref{prop:FenchelNielsen}] 

%
%

Since the representation restricted to each pant   is
irreducible,
it is locally parametrized by
the trace of its boundary curves. Namely, an irreducible character
in $SL_2(\mathbf C)$ of a free group with two generators $a$ and $ b$ is
parametrized by the traces of 
$a$, $b $ and $ab$  (see for instance \cite{FicoMontesinos}), that are precisely
the boundary curves of a pair of pants. We also use that since 
$\chi(\gamma_j)\neq\pm 2$, the value of a character at $\gamma_j$ is locally parametrized by 
$\mu_j$, because $\chi(\gamma_j)=2\cosh (\mu_j/2) \neq\pm 2$.
When the curve $\gamma_i$ is in a torus, the  cutoff of this component is a cylinder,
and its conjugacy class is parametrized by $\mu_j$.
Hence the $\mu_1,\ldots,\mu_{k}$ are
local coordinates for the restrictions to pants and cylinders. 
The coordinates are completed by adding the amalgamations along the curves $\gamma_i$,
namely the ${\lambda_1},\ldots,{\lambda_{k}}$.
\end{proof}

\begin{Definition}
The tangent vectors  $\{
\partial_{\mu_1},\ldots,\partial_{\mu_{k}},\partial_{\lambda_1},\ldots,
\partial_{\lambda_{k}}\}$
are the coordinate vectors of a parametrization as in
Proposition~\ref{prop:FenchelNielsen}.
\end{Definition}

Notice that $\{
\partial_{\mu_1},\ldots,\partial_{\mu_{k}},\partial_{\lambda_1},\ldots,
\partial_{\lambda_{k}}\}$ 
is a $\mathbf C$-basis for $H^1( \partial Z,Ad\rho)$.

Consider the pairing  that consists in combining the usual cup product
with the Killing form:
$$
B\!:\!\mathfrak{sl}(2,\mathbf C)\times \mathfrak{sl}(2,\mathbf C)\to\mathbf C,
$$
to get a 2-cocycle with values in $\mathbf C$, see 
 \cite{GoldmanAdvances,GoldmanInvariant,GoldmanCplx}.
We still denote by $\cup$ this paring:
\begin{equation}
\label{eqn:cupproduct}
\cup: H^1( \partial Z;Ad\rho)\times  H^1( \partial Z;Ad\rho) \to 
H^2( \partial Z;\mathbf C)\to \mathbf C.
\end{equation}
Here the last arrow is just the composition of the isomorphism 
$H^2( \partial_i Z,\mathbf C)\cong\mathbf C $
for each boundary component $\partial_i Z$ with the addition of the coordinates  
$\mathbf C\times\cdots\times\mathbf C\to\mathbf C$.

\begin{Theorem}[Goldman  \cite{GoldmanAdvances,GoldmanInvariant}]
\label{Theorem:symplectic}
The product (\ref{eqn:cupproduct}) defines a symplectic structure on $X(\partial Z)$.
Moreover $\partial_{\lambda_j}$ is the Hamiltonian vector field of 
 $\mu_j$:
$$
d\mu_j= \partial_{\lambda_j}\cup-.
$$
 \end{Theorem}

\begin{Corollary}
\label{lemma:cup}
Let $F$ be a function of $X(\partial M)$, and $H_F\in T_{\rho} X(\partial Z)$ its Hamiltonian,
the vector that satisfies $d F= H_F\cup-$. Then 
$$
d \mu_j (H_F)= -\frac{\partial F}{\partial \lambda_j}.
$$
%
%
%
\end{Corollary}

\begin{proof}
By  Theorem~\ref{Theorem:symplectic},
$$
d \mu_j (H_F)= \partial_{\lambda_j}\cup H_F= - H_F\cup \partial_{\lambda_j}
=-\frac{\partial F}{\partial \lambda_j}.
$$
\end{proof}

\begin{Theorem}[Duality Theorem]
\label{thm:duality}
Let $\chi\in X(Z)$ and $\Gamma=\{\gamma_1,\ldots, \gamma_{k}\}\subset \pi_1(\partial Z) $ 
satisfy the hypothesis of Proposition~\ref{prop:FenchelNielsen} and let 
$a_1,\ldots,a_{k}\in\mathbf C$.
There exists a tangent vector $v\in H^1(Z,Ad\rho)$ such that
$$
d\mu_i(v)= a_{i}, \qquad\textrm{ for }i=1,\ldots,k,
$$
if and only if $a_1d\lambda_1+\cdots+a_{k}d\lambda_{k}$ vanishes in the cotangent space to $X( Z,\Gamma)$.
\end{Theorem}

\begin{proof}
 Assume first that there exists a tangent vector $v\in H^1(Z,Ad\rho)=T_\chi X(Z)$ such that
$d\mu_i(v)=a_i$, for $i=1,\ldots, k$.
Let $i\!:\! \partial Z\to Z$, denote the inclusion, and 
$i^*\!:\! T_\chi X(Z)\to T_\chi X(\partial Z)$ 
the induced map in cohomology. 	Let $F$ be a function linear in $\mu_i$ and $\lambda_i$
 such that
$ i^* (v)=H_F$. Then by Corollary~\ref{lemma:cup}
$$
a_ j=d\mu_j (H_F)=-\frac{\partial F}{\partial \lambda_j}
$$
Hence 
$$
F = -a_1 \lambda_1-\cdots - a_{k} \lambda_{k}+ \sum b_i\mu_i,
$$
for some $b_i\in\mathbf C$.

For every $y\in T_\chi X(Z)$, 
$$
dF \circ i^*(y)=-i^*(y)\cup H_F= -i^*(y)\cup i^*(v)=0
$$
because the image of $i^*$ is an isotropic subspace.
Thus $i_*(dF)=0 $ and  by Lemma~\ref{lemma:tgrelative} (2),
$a_1 d\lambda_1+\cdots+ a_{k} d\lambda_{k}$ vanishes in the cotangent space 
$$
	 (T^{Zar}_{\chi_\rho})^* X( Z, \Gamma).
$$

To prove the converse, start assuming that $a_1 d\lambda_1+\cdots+a_{k} d\lambda_{k}$ vanishes in the cotangent space to
$X( Z, \Gamma)$. Thus $i_*(a_1 d\lambda_1+\cdots+a_{k} d\lambda_{k})\in \bigoplus_{i=1}^k H^1(\gamma_i;Ad\rho)$, by Lemma~\ref{lemma:tgrelative} (2). Hence
there exist $b_1,\ldots,b_{k}\in\mathbf C$ such that 
$$
 i_*(a_1 d\lambda_1+\cdots+a_{k} d\lambda_{k}) = i_*( -b_1\, d\mu_1-\cdots-b_{k}\, d\mu_{k}) \in H_1(\partial Z;Ad\rho).
$$
Setting $F=
-( a_1\lambda_1+\cdots+a_{k}  \lambda_{k}+ b_1 \mu_1+\cdots+b_{k}\mu_{k})
$
we have $i_*(dF)=0$. 
Working in cohomology, consider the image of
$$
i^*\! :\! H^1(Z;Ad\rho)\to  H^1(\partial Z;Ad\rho)
$$
which is a Lagrangian subspace of $H^1(\partial Z;Ad\rho)$, by a
standard argument using Poincar\'e duality. Moreover, since $i_*(dF)=0$, 
$dF$ is orthogonal to all deformations 
of $\partial Z$ induced from deformations of $Z$: 
$dF (\operatorname{Im}(i^*))=0$.
Let $H_F\in H^1(\partial Z;Ad\rho)$ be the ``Hamiltonian vector of $F$'':
$$
H_F\cup - = dF .
$$
In particular $H_F^\bot=\ker dF$ contains $\operatorname{Im}(i^*)$.
Since $\operatorname{Im}(i^*)$ is Lagrangian for the symplectic pairing,
$$
H_F\in \operatorname{Im}(i^*),
$$
otherwise $\operatorname{Im}(i^*)\oplus \langle H_F \rangle$ would contradict the maximality of
$\operatorname{Im}(i^*)$ among isotropic subspaces.

By Corollary~\ref{lemma:cup}:
$$
H_F=a_1 \partial_{\mu_1}+\cdots+ a_{k} \partial_{\mu_{k}} + \tilde b_1\,
\partial_{\lambda_1}+\cdots +\tilde b_{k}\, \partial_{\lambda_{k}},
$$
because $d\mu_j(H_F)=\frac{\partial F}{\partial \lambda_j}=a_j
$ (Corollary~\ref{lemma:cup}).
Notice that the
$\tilde b_j$ may be different from the $b_j$. 
As  $H_F\in \operatorname{Im}(i^*)$, there exists $v\in H^1(Z,Ad\rho)$ whose restriction to
$\partial Z$ is $H_F$, and therefore $d\mu_i(v)=d\mu_i(H_F)=a_i$, for
$i=1,\ldots,3n$.
\end{proof}

\section{All singular fibers are in the branching locus}
\label{section:adding}

In this section we make the following assumption, that we will remove in Section~\ref{section:curve}:

\begin{Assumption}
\label{Assumption:Ifibres}
 All singular $I$-fibers are in the branching locus of $\mathcal O^3$.
\end{Assumption}

Recall that $$
\chi_0\in X_{PSL_2(\mathbf C)}(\mathcal O^3)
$$ is the $PSL_2(\mathbf C)$-character induced by
the perimeter minimizing hyperbolic metric  of $P^2$.
Since all singular $I$-fibers are in the branching locus, we have:

\begin{Remark}
\label{rm:handlebody} 
The smooth part 
$$
H=\mathcal O^3\setminus \mathcal N(\Sigma_{\mathcal O^3}),
$$
 is a handlebody 
of genus $n+1$.
\end{Remark}

Consider 
$$
\Gamma=\{\gamma_1,\ldots, \gamma_{3n}\}\subset \pi_1( \partial H)
$$
the (oriented) meridian curves for $\mathcal O^3$, one for each singular arc of  $\Sigma_{\mathcal O^3}$.
In particular they give a pants decomposition of $\partial H$.
Order them so that:
\begin{itemize}
 \item 
$
\Upsilon=\{\gamma_{2n+1},\ldots \gamma_{3n}\}
$
is the set of \emph{vertical} meridians (around $\Sigma_{\mathcal O^3}^{Vert}$), and
\item $\Gamma\setminus \Upsilon$ is the set of \emph{horizontal} meridians (around  $\Sigma_{\mathcal O^3}^{Hor}$).
\end{itemize}
 We have the following isomorphisms
\begin{equation}
 \label{eqn:XF}
X_{PSL_2(\mathbf C)}^{irr}(P^2)\cong 
X_{PSL_2(\mathbf C)}^{irr}(\mathcal O^3)
\cong
X_{PSL_2(\mathbf C)}^{irr}(H, \Gamma)\cong_{LOC} X^{irr}(H, \Gamma).
\end{equation}
The first isomorphism is Lemma~\ref{Lemma:basefibr},
the second one is Lemma~\ref{lemma:isoorbifoldrel}, and the third one follows from the fact that all representations of a free group to
$PSL_2(\mathbf C)$ lift to $SL_2(\mathbf C)$.

We have an inclusion 
$$
X(H,\Gamma)\subset X(H,\Upsilon).
$$

In a neighborhood of $  U\subset X(H, \Upsilon)$ of $\chi_N$, we define:
$$
 \mu=(\mu_1,\ldots,\mu_{2n}) : U\subset X(H, \Upsilon) \to  \mathbf C^{2n}
$$
so that 
$$
X(H,\Gamma)\cap U= \mu^{-1}(\pi\mathbf i,\ldots, \pi\mathbf i).
$$

\begin{Lemma}
\label{lemma:normal}
\begin{enumerate}
 \item The pair $(U,X(H,\Gamma)\cap U)$ is biholomorphic to a neighborhood of the origin in in $(\mathbf C^{2n},\mathbf C^{n-3})$
\item The tangent map 
$$
\mu_*:T_{\chi_N} U\to T_{(\pi\mathbf i,\ldots, \pi\mathbf i)}\mathbf C^{2n}
$$
is injective on the normal bundle to $X(H,\Gamma)\cap U$.
\end{enumerate}
 \end{Lemma}

\begin{proof}
We first prove that $U$ is biholomorphic to a neighborhood of $\mathbf C^{2n}$.
The 2-orbifold $P'=P^2\setminus vertices(P^2)$ obtained by removing the vertices of $P^2$
can be deformed by changing the angles of the vertices. This gives  
$\nu_1,\ldots, \nu_{ n}$ tangent vectors to the variety of characters of
$P^2\setminus vertices(P^2)$ in $PGL_2(\mathbf R)$, 
one for each cone angle (keeping the other angles fixed). Let $\bar \nu_i$
denote the corresponding vectors in the variety of characters of $H$.
The trace functions of $\gamma_{2n+i}$ satisfy
$
d\, \operatorname{Trace}_{\gamma_{2n+i}}(\bar \nu_j)=\delta_{ij}
$,
which  implies that 
$
X( H,\Upsilon)
$
is smooth at $\chi_N$ and has dimension $2n$. 

Using elementary hyperbolic trigonometry, one can prove that the Teichm\"uller space of $P^2$ (the space of $n$-polygons in $\mathbf H^2$ with fixed angles) embeds in $\mathbf R^n$,
with coordinates edge lengths, and it is a smooth submanifold of codimension $3$. 
Let $V\subset\mathbf C^n$ be the complexification of this normal space,
and, assuming that $\gamma_i$ and $\gamma_{n+i}$ project to the same edge of $P^2$, $i=1,\ldots,n$, let 
$$
V'=\{(a_1,\ldots,a_{2n})\in \mathbf C^{2n}\mid (a_1+a_{n+1},\ldots , a_n+a_{2n})\in V\}\cong \mathbf C^{n+3}. 
$$
Thus, if $(a_1,\ldots,a_{2n})\in V'$, then $a_1d \lambda_1+\cdots + a_{2n} d\lambda_{2n}$ vanishes in the cotangent space
$$(T^{Zar}_{\chi_0})^* X_{PSL_2(\mathbf C)}(P^2)\cong (T^{Zar}_{\chi_0})^* X(H,\Gamma).$$
By the duality theorem (Thm.~\ref{thm:duality}), $V'$ is contained in the image of $\mu_*:T_{\chi_0} U\to T_{(\pi\mathbf i,\ldots, \pi\mathbf i)}\mathbf C^{2n}$,
ie.\ this map has rank at least $n+3$. Since the dimension of $X_{PSL_2(\mathbf C)}^{irr}(P^2)\cong_{LOC} X^{irr}(H, \Gamma)$  is $n-3$, the lemma follows.
\end{proof}

Write $V=\mu (U)\subset \mathbf C^{2n}$. Let $\check U$ be the blow-up of $U$ at the submanifold  $X(H,\Gamma)\cap U$ and $\check V$ the blow-up of $V$ at the point
$(\pi\mathbf i,\ldots, \pi\mathbf i)\in V$. The respective exceptional divisors are denoted by $E_U\subset \check U$ and $E_V\subset \check V$.

\begin{Lemma}
 \label{lemma:blowuplift}
The map $\mu$ lifts to the blow-up, so that the following diagram commutes:
$$\xymatrix{
  (\check U, E_U)\ar[r]^{\check\mu} \ar[d]_{pr_U} &  (\check V, E_V) \ar[d]^{pr_V} \\
(U,X(H,\Gamma)\cap U)  \ar[r]^\mu & (V, (\pi\mathbf i,\ldots,\pi\mathbf i)
 }
$$

\end{Lemma}

\begin{proof}
This is a consequence of  Lemma~\ref{lemma:normal} (2), that applies no only to $\chi_0$ but to a neighborhood of it
in $X(H,\Gamma)$, because that condition to lift is that $\mu_*$ is injective on the normal bundle.
\end{proof}

\begin{Lemma}
\label{lemma:isolated} The character 
$\chi_0$ is an isolated critical point  in 
$X_{PSL_2(\mathbf C)}(P^2)$
 of the complex function
$\lambda_1+\cdots+\lambda_{2n}$. In addition, for any choice of local coordinates, the determinant of the Hessian at  $\chi_0$ does not vanish.
\end{Lemma}

\begin{proof}
By Kerckhoff's proof of Nielsen conjecture \cite{KerckhoffAnnals}, 
it is an isolated critical point when restricted to the Teichm\"uller space
of $P^2$ (after taking a finite covering that is a manifold). As explained in \cite{KerckhoffAnnals},
the second derivative is nonzero in all directions tangent to earthquake
deformations, by using the estimates of Wolpert \cite{Wolpert} on the second derivative
of twist deformations (See Appendix~\ref{sec:appendix}). By \cite[Thm~3.5]{KerckhoffDuke}, every tangent direction
in Teichm\"uller space is tangent to an earthquake. Hence the restriction of
$\lambda_1+\cdots+\lambda_{2n}$ to the Teichm\"uller space has a nondegenrate critical
point at $\chi_0$ (ie.\ its Hessian is positive definite). 
By complexifying, it follows that it is an isolated critical point on
quasifuchsian space.
\end{proof}

By using  the duality theorem (Thm.~\ref{thm:duality}), and since $\chi_0$ minimizes the perimeter of $P^2$, we can make the following definition:

\begin{Definition}
We denote by $v_0\in T_ {\chi_N} X(H,\Upsilon)$ a vector that satisfies
$$
 d\mu_i(v_0)= \left\{
\begin{array}{l}
1,\textrm{ for }i=1,\ldots, 2n, \\
0,\textrm{ for }i=2n+1,\ldots, 3n.
\end{array}\right.
$$
\end{Definition}

Elements of $E_U$ are directions of vectors $v$ normal to $X(H,\Gamma)$, denoted by $\langle v\rangle$.

\begin{Proposition}
\label{Prop:blowupbihol}
The map $\check\mu$ restricts to a biholomorphism between a neighborhood of  $\langle v_0\rangle $
in $\check U$ and a neighborhood of $\langle (1,\ldots,1)\rangle $ in $\check V$.
\end{Proposition}

\begin{proof}
We prove first that $\check\mu$  is injective in  a neighborhood of $\langle v_0\rangle $   in the exceptional divisors $E_U$.
By Lemma~\ref{lemma:isolated}, the determinant of the Hessian of $\lambda_1+\cdots+\lambda_{2n}:X(H,\Gamma)\cap U\to \mathbf C$ at $\chi_0$ does not vanish. Hence
by the implicit function theorem, for any $(a_1,\ldots,a_{2n})\in \mathbf C^{2n}$ in a
 neighborhood of $(1,\ldots, 1)$, there exists a unique $\chi\in X(H,\Gamma)\cap U$ such that $\chi$ is a critical point of $a_1d\lambda_1+\cdots + a_{2n}d\lambda_{2n}$.
In particular, by Theorem~\ref{thm:duality} there exists a unique $\chi\in X(H,\Gamma)\cap U$ such that its tangent space contains a vector $v\in T_{\chi} U$ with $\mu_*(v)=(a_1,\ldots,a_{2n})$. 
Moreover, this $v$ is unique in the normal bundle, by Lemma~\ref{lemma:normal} (2). This proves that 
 $\check\mu$ 
is injective in a neighborhood  in the exceptional divisor $E_U$. By holomorphicity, this implies that $\check\mu$ is a biholomorphism between the 
neighborhoods in $E_U$ and $E_V$. By construction $\check\mu_*$ is injective in the normal direction to $E_U$ ($(pr_U)_*$ is injective in the normal direction), hence the inverse function theorem applies.
\end{proof}

\begin{Corollary}
 \label{cor:curve}
There exists an algebraic $\mathbf C$-curve $\mathcal C\subset X(H;\Upsilon)$ containing $\chi_N$, such that 
$\mu\vert_{ \mathcal C}$ is  a biholomorphism between a neighborhood of $\chi_N$ and a neighborhood of the diagonal
 $\mu_1=\cdots=\mu_{2n}$.
\end{Corollary}

\begin{Remark}
In this way we already obtain a path of representations analogue to Proposition~\ref{assumption:rho}, just by considering the path 
$\mu_1=\cdots=\mu_{2n}=\mathbf i (\pi-t)$. Lifting it to a deformation of representations $\rho_t$, 
and since $\rho_t(\gamma_i)\in\mathbf R$,
the deformation satisfies, up to conjugation
\begin{equation}
\label{eqn:rhomi} 
\rho_t(\gamma_i)=\pm 
\begin{pmatrix} e^{\mathbf \mu_i(t)/2} & 0 \\
	0 & e^{-\mathbf \mu_i(t)/2} 
\end{pmatrix}.
\end{equation}
\end{Remark}

\begin{Remark}
\label{remark:changesign}
 Replacing $t$ by $-t$ in the previous choice changes the sign of the trace. This corresponds to changing the orientation 
because when we take the complex conjugate, the sign of the trace of $\rho_t(\gamma_i)$ in Equation~(\ref{eqn:rhomi})
is changed, but also the sign of $$
\rho_0(\gamma_i)=\pm \begin{pmatrix} e^{\mathbf i\pi/2} & 0 \\
	0 & e^{- \mathbf i\pi/2} 
\end{pmatrix},
$$
hence the sign in the relation $tr(\rho_t(\gamma_i))=\pm 2\cos(\alpha_i(t)/2)$.
\end{Remark}

\section{Constructing a curve of representations of $M$}
\label{section:curve}

The goal of this section is to prove the following proposition, which implies Proposition~\ref{assumption:rho}. We will use the previous section and a deformation argument.

\begin{Proposition}
\label{prop:defmM}
There exists an algebraic curve of representations  of $M=\mathcal O\setminus \mathcal N(\Sigma_{\mathcal O})$ containing $\chi_0$, so that  all
meridians can be deformed by decreasing their rotation angle, 
and the cone angle is the same for each meridian.
\end{Proposition}

\begin{proof}
The case where all singular fibers are contained in $\Sigma_{\mathcal O}$ is discussed in the previous section. 
To simplify, we assume that 
$\Sigma_{\mathcal O}$ is a link (ie.\ $\Sigma_{\mathcal O}^{Vert}=\emptyset$), so that we deform the angle of all singular $I$-fibers.
 In the general case we should only deform some of the $I$-fibers,

For $N>1$, let $\mathcal O^3_N$ denote the orbifold obtained by adding a label $N$ to all singular fibers. Thus $\mathcal O^3_N$ is an orbifold 
that satisfies Assumption~\ref{Assumption:Ifibres}.
The hyperbolic structure on the basis is modified, and the angles of $P^2$ are divided by $N$, obtaining a new hyperbolic structure.
The new orbifold is denoted by $P^2_N$ and the character of the 
hyperbolic structure that minimizes the perimeter 
is denoted by $\chi_N$.

Set $H=\mathcal O^3_N\setminus \Sigma_{\mathcal O^3_N}$.
Recall that  $\Upsilon=\{\gamma_{2n+1},\ldots,\gamma_{3n}\}\subset \pi_1(H)$ denotes the set of 
vertical meridians, ie.\ meridians of the singular
components of $\mathcal O_N$ corresponding to singular $I$-fibers,
and  $\Gamma\setminus\Upsilon=\{\gamma_1,\ldots \gamma_{2n}\}\subset \pi_1(H)$ denote the set of horizontal ones.

The character $\chi_N$ satisfies $\chi_N(\gamma_i)=0$ for $i=1,\ldots,2n$. 
The characters of the curve of Corollary~\ref{cor:curve} with all cone angles equal
 satisfy
$\chi(\gamma_i)=\pm \chi(\gamma_j)$, and the sign depends on the lift of 
the holonomy of $ P_N^2$ to $SL_2(\mathbf C)$.
 Namely, 
a rotation of angle $\pi$ that fixes the oriented axis in 
the upper half space model for $\mathbf H^3$ that goes from $0$ 
to $\infty$ is
$$
\pm \begin{pmatrix} 
	\mathbf{i} & 0 \\
	0 & -\mathbf{i}
\end{pmatrix}= 
\pm
 \begin{pmatrix} 
	e^{\mathbf{i}\pi/2} & 0 \\
	0 & e^{-\mathbf{i}\pi/2}
\end{pmatrix}.
$$
Thus decreasing the angle $\pi$ affects differently 
the sign of the trace: since we work with 
half angle, it depends on whether we start with 
$\pi/2$ or $3\pi/2$. In what follows, we will assume that for infinitely many 
$N$, $\chi(\gamma_i)= \chi(\gamma_j)$, ie.\ we are able to make the same choice of lift for
for infinitely many $N$. Otherwise, some equalities $\chi(\gamma_i)=\chi(\gamma_j)$
have to be replaced by $\chi(\gamma_i)=-\chi(\gamma_j)$.

Let $S_{\mathbf R}$ and $S_{\mathbf C}$ denote
the respective $\mathbf R$ and $\mathbf C$-Zariski closures in $X(H)$
of the union of curves provided by Corollary~\ref{cor:curve}.
Using the results of Section~\ref{section:developingmaps}, they correspond to hyperbolic cone manifolds,
and rigidity results apply.
By local rigidity of hyperbolic cone manifolds,  $S_{\mathbf C}$ is a
 $\mathbf C$-irreducible component of the set defined by the
equations:
$$
\left\{
\begin{array}{ll}
 \chi(\gamma_{2n+i})=\chi(\gamma_{2n+j}) & \textrm{ for } i,j=1,\ldots n, \textrm{ (ie.\ }\gamma_{2n+i},\gamma_{2n+j}\in\Upsilon);\\
 \chi(\gamma_i)=\chi(\gamma_j) & \textrm{ for } i,j=1,\ldots 2n.
\end{array}
\right.
$$
Using this local rigidity, 
 $S_{\mathbf C}$ is a surface, because $\chi(\gamma_{2n+1})$ is one of the parameters, and
$\chi(\gamma_1)$ is the second parameter.
Taking  real values for these parameters, we obtain the  real surface $S_{\mathbf R}$.

Set the orbifold $P'=P^2\setminus vertices(P^2)$. The real
surface $S_{\mathbf R}$ intersects $X(P')$ in infinitely many points (infinitely many odd $N$),
hence 
there is a component  of
$S_{\mathbf R}\cap X(P')$ that  is a real curve  and contains all the characters
corresponding to $\chi_N$, for infinitely many $N$.
We shall show in Corollary~\ref{coro:chi0inS} that  $S_{\mathbf R}\cap X(P')$ contains $\chi_0$.

If $q_1,\ldots,q_n$ denote the indices of the singular $I$-fibers, then
 the angles of the vertices of $P^2$ are $\pi/q_1,\ldots,\pi/q_n$,
so that the angles of the vertices of $P^2_N$ are $\pi/(N q_1),\ldots,\pi/(N q_n)$.

For $0<t\leq 1$, 
let $\mathcal T_t$ denote the \emph{Teichm\"uller space} of polygons with
given angles $t\frac{\pi}{q_1},\ldots,t\frac{\pi}{q_n}$, and 
$\mathcal {QF}_t$,  the \emph{quasifuchsian space} of polygons with those angles,
which is locally the 
 complexification of $\mathcal T_t$.

\begin{Lemma}
\label{lemma:compact}
On every $\mathcal T_t$ there exists a unique minimizer of the perimeter. In
addition, the perimeter has a non degenerate critical point of $\mathcal {QF}_t$.
\end{Lemma}

Lemma~\ref{lemma:compact} is the analogue of Lemma~\ref{lemma:isolated} and a consequence of the earthquake theory and the results of
Kerckhoff and Wolpert in this setting. This will be explained and proved in
Appendix~\ref{sec:appendix}, Corollary~\ref{cor:minimizer}.

\begin{Lemma}
\label{lemma:morecritical} 
For $0<t\leq 1$, $S_{\mathbf R}\cap \mathcal {QF}_t$ is contained in the critical set of 
 $\lambda_1+\cdots+\lambda_{2n}$ restricted to $\mathcal {QF}_t$.
\end{Lemma}

\begin{proof}
Let $\chi'\in S_{\mathbf R}\cap \mathcal {QF}_t$.
Set $P'=P^2\setminus vertices(P^2)$.
 Since $S_{\mathbf R}$ is an irreducible  surface,  $S_{\mathbf R}\cap X(P')$
is a curve and therefore there exist a curve that deforms $\chi'$ in $S_{\mathbf R}$ away from  $X(P')$.
We lift this curve from  the variety of characters
to the variety of representations. 
We obtain in this way an analytic path of representations $\rho_s'$ in $\mathcal
C'$, with $\rho_0'=\rho'$ a representation whose character is $\chi_{\rho'}=\chi'$.
Let $l\geq 0$ be maximal such that the power expansion
$$
\rho_s'(\gamma)=(1+ s a_1(\gamma)+ \cdots   + s^l a_l(\gamma) +
s^{l+1} a_{l+1}(\gamma) +\cdots )\rho'(\gamma),
\qquad \forall \gamma\in\pi_1(H),
$$
is a representation  in $SL_2(\mathbf C[s]/(s^{l+1}))$
that factors through $\pi_1(P')$, but as a representation in 
$SL_2(\mathbf C[s]/(s^{l+2}))$ does not factor
through $\pi_1(P')$.

Namely, up to conjugation we may assume that 
 $a_i\!:\!\pi_1(H)\to M_2(\mathbf C)$ factors through $\pi_1(P')$, for
$i=1,\ldots, l$, but  $a_{l+1}$ does not factor. 
Since the variety of representations 
of $P'$ is smooth, there exists $b:\pi_1( H)\to M_2(\mathbf C)$ such that 
$$
\gamma\mapsto (1+ s a_1(\gamma)+ \cdots   + s^la_l(\gamma)  + s^{l+1}
b(\gamma))\rho'(\gamma)
$$
is a representation of $P^2_N$ in $SL_2(\mathbf C[s]/(s^{l+2}))$.
The compatibility relations to be
a representation imply that 
$$
a_{l+1}-b
$$
is a group cocycle of $\pi_1(H)$ taking values in the Lie algebra
$\mathfrak{sl}_2(\mathbf C)$. 
In addition by maximality of $l$, $a_{l+1}-b$ is nontrivial on horizontal meridians, and by
construction it is tangent to
all meridians being equal. 
Hence we obtain a cohomology element $v_0=[a_{l+1}-b]\in H^1(H;Ad \rho')$ that satisfies
$$
d\mu_i(v_0)=
1,\qquad \textrm{ for } i=1,\ldots,2n.
$$
Considering deformations of $P'=P^2\setminus vertices(P)$, there exist a tangent vector
$v_1\in H^1(H;Ad\rho')$ such that $d\mu_j(v_1)=0$, $j=1,\ldots,2n$, (because it is obtained from deformations of $P'$)
and $d\mu_{2n+j}(v_1)=d\mu_{2n+j}(v_0)$, $j=1,\ldots,n$ (by perturbing the angles). Thus  $v=v_0-v_1$ satisfies 
$$
d\mu_j(v)=\left\{
\begin{array}{l}
 1, \textrm{ for  } j=1,\ldots,2n; \\
 0, \textrm{ for  } j=2n+1,\ldots,3n.
\end{array}\right.
$$
Then we apply Theorem~\ref{thm:duality}, 
and the lemma follows from the local isomorphism
between 
$X(H,\Gamma)$ and $X(P',\Upsilon)$,
by combining Lemmas~\ref{lemma:isoorbifoldrel} and \ref{Lemma:basefibr}.
 \end{proof}

\begin{Corollary}
\label{coro:chi0inS} The curve
$S_{\mathbf R}\cap X(P')$ contains $\chi_0$.
\end{Corollary}

\begin{proof}
By Corollary~\ref{cor:curve}, for infinitely many  natural $N$, the Kerckhoff minimizer
of Lemma~\ref{lemma:compact}, 
$\tau_{1/N}\in  \mathcal T_{1/N}$, is contained in $S_{\mathbf R}\cap X(P')$. Thus 
there is an irreducible $\mathbf R$-curve $\mathcal D\subset  S_{\mathbf R}\cap X(P')$
that contains infinitely many $\tau_{1/N}$. 
Define 
$$
I=\{
t\in (0,1]\mid \tau_t\in \mathcal D
\},
$$
where $\tau_t$ denotes the Kerckhoff minimizer  of Lemma~\ref{lemma:compact}.
We claim that $1\in I$. We use
a connectedness argument.
Since $1/N\in I$ for infinitely many $N$, $I\neq\emptyset$. 
The set of Kerckhoff minimizers $\tau_t$ 
is closed in $\cup_t\mathcal T_t$, and so it is in $\mathcal D$.
Hence $I$ is closed. For openness, we have:
\begin{enumerate}
 \item The set $\mathcal D\cap \bigcup_{t\in (0,1)}\mathcal {QF}_t$ is open in $\mathcal D$, 
in particular it is locally an algebraic curve. This follows from the fact that the quasifuchsian space 
is an open subset of the variety of characters (and using the corresponding restrictions on the cone angles).
\item By  Sullivan's theorem \cite{Sullivan}, the Euler characteristic of the link  of
any point  in $\mathcal D$ (hence in $\mathcal D\cap \bigcup_{t\in (0,1)}\mathcal {QF}_t$) is even.
 \item For every $t\in (0,1]$ one of the components of the intersection 
$S_{\mathbf R}\cap \mathcal {QF}_t $ is an isolated point, 
precisely equal to 
$\tau_t$,
by Lemmas~\ref{lemma:compact} and \ref{lemma:morecritical}.
\end{enumerate}

These three facts  imply that $\mathcal D\cap \mathcal {QF}_{t'}\neq\emptyset $
for $t'$ in a neighborhood of $t\in (0,1)$.
By Lemma~\ref{lemma:morecritical}, this intersection must be 
precisely equal to the minimizer $\tau_{t'}$.
\end{proof}

Consider $\mathcal E$  an irreducible component of 
$$
	\{ \operatorname{Trace}_{\gamma_{2n+1}}=2\}\cap S_{\mathbf C}
$$
that contains $\chi_0=\chi_{\rho_0}$. Since the intersection is nonempty (it contains $\chi_0$)
and it is not the whole $S_{\mathbf C}$ (the angles are not constant in $S_{\mathbf C}$), it is 
a complex curve.

\begin{Lemma}
\label{claim:factorstom}
 If $\rho\in R(H)$ is a representation close to $\rho_0$ and $\chi_{\rho}\in \mathcal E$, 
then $\rho(\gamma_{2n+i})$ is the identity matrix for all $i=1,\ldots, n$. In particular it factors to
a representation of $M$.
\end{Lemma}

\begin{proof}
 The element $\gamma_{2n+1}$ is the meridian of a singular $I$-fiber of $\mathcal O_ N$.
Each endpoint of this edge 
meets the endpoints  of two more branching edges of $\mathcal O^3_ N$, with
respective meridians
 $\varsigma$ and $\tilde\varsigma$ in $\pi_1(H)$.
They satisfy $\varsigma \gamma_{2n+1}=\tilde\varsigma$ and $\varsigma$ and
$\tilde\varsigma$
project both to the same element $\sigma_1$ in $\pi_1(M)$ (using the notation of
Section~\ref{section:fibration}, cf.~Figure~\ref{fig:loopstangle}).
Since $\rho$ is close to $\rho_0$ and $\rho_0(\sigma_1)$ is a rotation of angle
$\pi$, we may assume that
$\rho(\varsigma)$ and $\rho(\tilde\varsigma)$ are both diagonal matrices with
(equal) eigenvalues 
$\lambda^{\pm 1}\neq \pm 1$. 
We write
$$
\rho(\varsigma)=\begin{pmatrix}
                 \lambda & 0 \\ 0 & \lambda^{-1} 
                \end{pmatrix}
\textrm{ and }
\rho(\gamma_{2n+1})=\begin{pmatrix}
                 a & b \\ c & d 
                \end{pmatrix},
$$
with $ad-bc=1$, $a+d=2$. Since $\chi_{\rho}\in S_{\mathbf C}$,
 $a \lambda+d\lambda^{-1}=\lambda+\lambda^{-1}\neq\pm 2$. Thus $a=d=1$
and either $b$ or $c$ vanishes. 
This means that if $ \rho(\gamma_{2n+1})$ is not the identity but  parabolic, 
then the fixed point of $
\rho(\gamma_{2n+1})$ 
has to be one of the endpoints of the axis of $\rho(\varsigma)$. 
Let $\sigma_1'$ be the meridian of the opposite edge in the tangle, so that the tangle group is the free group on $\sigma_1$ and $\sigma_1'$.
The axis of
$\rho_0(\sigma_1)$ and $\rho_0(\sigma_1')$ form an angle,
hence the endpoints of their axis are far, and
the previous argument  for $\sigma'_1$ instead of $\sigma_1$ gives a contradiction with the hypothesis 
that $\rho(\gamma_{2n+1})$ is not the identity.
\end{proof}

\begin{Claim}
 \label{claim:munonconstant}
The trace of the meridian $\gamma_1$  is not constant along  $\mathcal E $.
\end{Claim}

\begin {proof}
By contradiction, assume that  it is constant, then 
$\mathcal E$ 
is contained in $X(P')$. 
Take a character $\bar \chi \in \mathcal E$ close to $\chi_0$. 
Lemma~\ref{claim:factorstom}   implies that 
$\bar\chi$ induces a character of $M$ and of $ \mathcal O^3$, in particular it lies
in $X(P^2)$. 
Since $\bar \chi\in\mathcal E $ but $S_{\mathbf C}$ is not
contained in $X(P')$, the argument in the proof 
of Lemma~\ref{lemma:morecritical} implies that there is a tangent vector
$v\in T_{\bar \chi} X(H)$ that satisfies $d\mu_i(v)=1$, for $i\leq 2n$. 
In addition, as in Lemma~\ref{lemma:morecritical}, 
the restriction of $v$ to each $\gamma_{2n+i}$ can be made
zero by adding
infinitesimal deformations of $P'$, hence $v\in T_{\bar \chi}  X(M)$. 
By Theorem~\ref{thm:duality}, $\bar \chi$ is a critical point of the perimeter
in
$$
X^{irr}(M,\Gamma')\cong_{LOC} X^{irr}_{PSL_2(\mathbf C)}( \mathcal O^3)\cong X^{irr}_{PSL_2(\mathbf C)}( P^2),
$$
where $\Gamma'\subset\pi_1(M)$ is a collection of meridians for the singular components of $\mathcal O^3$.
This contradicts the analogue of 
 Lemma~\ref{lemma:isolated}, that the Kerckhoff minimizer is an isolated critical point
of the perimeter in $X_{PSL_2(\mathbf C)}(P^2)$.
\end{proof}

\noindent\emph{End of the proof of Proposition~\ref{prop:defmM}.} 
By Lemma~\ref{claim:factorstom}, $\mathcal E $ gives a
curve of representations of $M$.
In addition, by Claim~\ref{claim:munonconstant}, the trace of the meridian on
this curve is nonconstant.
A nonconstant complex map is open, thus by looking at the inverse image of
points with real trace, we find the
path of representations we are looking for.
\end{proof}

\begin{Remark}
Once we have Proposition~\ref{prop:deformreps} below, the trace of the meridian on ${\mathcal E}$ is
a local diffeomorphism around $\chi_0$.
\end{Remark}

In fact, the trace of the meridian on ${\mathcal E}$ 
 cannot be a ramified covering, because this would
contradict global rigidity of hyperbolic cone manifolds. Namely, we only can have one inverse image of the
real line, that gives two branches,
 corresponding to the two complex  conjugate representations, one with trace
$2\cos(\alpha)·$ and the other one with 
 trace $-2\cos(\alpha)$.

\section{The fibration of the orbifold}
\label{section:fibration}

The orbifold $\mathcal O^3$ is Seifert fibered over $P^2$:
$$
 S^1\to \mathcal O^3\overset{p}\to P^2.
$$

We distinguish three kinds of points of $P^2$: interior points of the underlying space
$\vert P^2 \vert$, interior
points of the mirror edges, and vertices.
Each interior point of $P^2$ has a neighborhood $U$  such that $p^{-1}(U)$
is a fibered solid torus.
By hypothesis, there is at most one cone point in the interior.
Such a point has a neighborhood $U\subset P^2$,
such that $p^{-1}(U)$ can have a singular core, a singular Seifert
fibration, or both. By Remark~\ref{remark:nointeriorpoint}, we may assume that there is no such interior cone point.
Points in the  boundary of $\vert P^2 \vert$ have a neighborhood with inverse image an orbifold with
topological underlying space
 a ball, and with branching locus two unknotted arcs of order $2$, possibly linked by a segment, giving a graph with $H$-shape.
For points in the interior of the edges, the fibration is nonsingular, but for
vertices,
the fibre is either singular, or in the branching locus, or both. The singularity and the  branching determine  the angle, see
\cite{BonahonSiebenmann}.
More precisely, there is a  rational number $p/q\in\mathbf Q$, $p,q\in\mathbf Z$
coprime, 
describing the singular fibration, and the angle at the vertex of $P^2$
is $\pi/(mq)$, cf.\ Figure~\ref{fig:fibration}, where $m\geq 1$ is the branching index (not branched for $m=1$).

\begin{figure}
\begin{center}
{
\psfrag{p1}{$p_1/q_1$}
\psfrag{p2}{$p_2/q_2$}
\psfrag{p3}{$p_3/q_3$}
\psfrag{p4}{$p_4/q_4$}
\includegraphics[height=4cm]{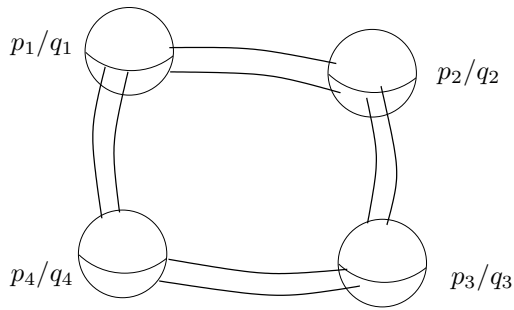}
}
\end{center}
   \caption{Fibration with 4 vertices The $p_i/q_i$-tangles around the 
$I$-singular fibers are inside the balls of the picture}\label{fig:fibration}
\end{figure}

We  \emph{orient} the components of  $\Sigma_{\mathcal O^3}^{Hor}$. The 
fiber of the interior of each edge of $ P^2$ contains two subsegments 
of $\Sigma_{\mathcal O^3}^{Hor}$, that project homeomorphically to the edge. 
The segments of $\Sigma_{\mathcal O^3}^{Hor}$ may induce the same or opposite 
orientations.

\begin{Remark} 
\label{remark:orient}
The  orientations induced by  $\Sigma_{\mathcal O^3}^{Hor}$ 
can be chosen to be either compatible for every edge of $P^2$,
or opposite for every edge.
 \end{Remark}

It suffices to prove this remark when  $\Sigma_{\mathcal O^3}$ is a link.
Notice that when at least one of the indices $q_i$ of the $I$-singular fibers of the vertices is even,
then the orientations of all pairs of edges are opposite. When all singular indices  are odd and
$\Sigma_{\mathcal O^3}$ is connected, then the orientations are compatible. 
Finally, when all singular indices are odd and $\Sigma_{\mathcal O^3}$ is not connected, then
$\Sigma_{\mathcal O^3}$ has two components and the orientations can be chosen compatible or 
opposite.

\medskip

Set
$$
M=\mathcal O^3\setminus\mathcal N(\Sigma_{ \mathcal O^3 }).
$$
We choose elements and subgroups of the fundamental group of $M$ according to
the fibration.
 We fix a base point $x_0\in M$  that projects to an interior point of $P^2$.

Let $n$ denote the number of vertices of $P^2$. 
The vertices of $P^2$ are denoted by $v_1,\ldots,v_n$,  and the edges,
$e_1,\ldots,e_n$, so that the endpoints of $e_i$
are $v_i$ and $v_{i+1}$, with coefficients modulo $n$. We distinguish the
following elements of $\pi_1(M,x_0)$:

\begin{itemize}
\item Let $f\in\pi_1(M,x_0)$ be an element represented by the fiber through $x_0$.
In particular $f$ projects to  the center of an  index two subgroup of $\pi_1(\mathcal O^3)$.

\item
For each edge $e_i$ of $P^2$, let  $e_i$ and $e_i'$ denote still the components of the singular locus of $\mathcal O^3$ that project to it.
We choose meridians $m_i$ and $m_i'$ by joining
$x_0$ to $e_i$  and $e_i'$ along a path that projects to an interior path of $P^2$, and then turn around the respective axis, so that 
$m_i$ and $m_i'$ differ only in a neighborhood of the  $I$-fiber.
We orient $m_i$ and $m_i'$ accordingly to the orientation of the edges.
Thus, when the orientations of the edges are compatible, 
we
require that 
$$m_i m_i'=f,$$
 (Figure~\ref{fig:loops}). When they are opposite, 
$$
m_i (m_i')^{-1}=f.
$$

\item For each vertex $v_i$ of $P^2$ we choose $V_i$ a neighborhood of the corresponding singular $I$-fiber and we call
$\pi_1(V_i\setminus (\Sigma_{\mathcal O}\cap V_i))$ the \emph{$i$-th tangle group}. We distinguish two cases.

If the singular $I$-fiber is not in the branching locus of the orbifold, then the tangle group is the free group on two meridians 
$\sigma_i,\sigma_i'\in \pi_1(M)$. We
choose a point 
in the middle of the singular fiber, and from there we consider both loops
(Figure~\ref{fig:loops}).

When the  singular $I$-fiber is in the branching locus of the orbifold, the tangle group is isomorphic to the fundamental group of a sphere with 
4 punctures. We choose generators $\varsigma_i,\bar \varsigma_i,\varsigma_i',\bar \varsigma_i' \in\pi_1(V_i\setminus (\Sigma_{\mathcal O}\cap V_i)) $
such that $\varsigma_i\bar\varsigma_i^{-1}=(\varsigma_i')^{-1}\bar\varsigma_i'$ is a meridian for the singular $I$-fiber. We choose the loops similarly
(Figure~\ref{fig:loopstangle}).

\end{itemize}

\begin{figure}
\begin{center}
{\psfrag{mi}{$m_i$}
\psfrag{mi'}{$m_i'$}
\psfrag{ei}{$e_i$}
\psfrag{ei'}{$e_i'$}
\psfrag{si}{$\sigma_i$}
\psfrag{si'}{$\sigma_i'$}
\psfrag{x0}{$x_0$}
\includegraphics[height=2.5cm]{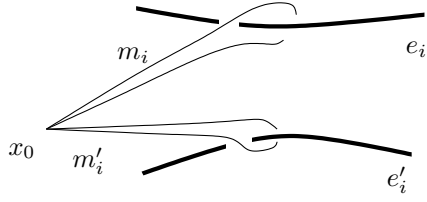}
}
\end{center}
   \caption{The loops for the meridians around regular $I$-fibers ($i$-th edge) )}\label{fig:loops}
\end{figure}

\begin{figure}
\begin{center}
{\psfrag{v1}{$\varsigma_i$}
\psfrag{v2}{$\bar \varsigma_i$}
\psfrag{v3}{$\varsigma_i'$}
\psfrag{v4}{$\bar\varsigma_i'$}
\psfrag{si}{$\sigma_i$}
\psfrag{si'}{$\sigma_i'$}
\psfrag{x0}{$x_0$}
\includegraphics[height=3cm]{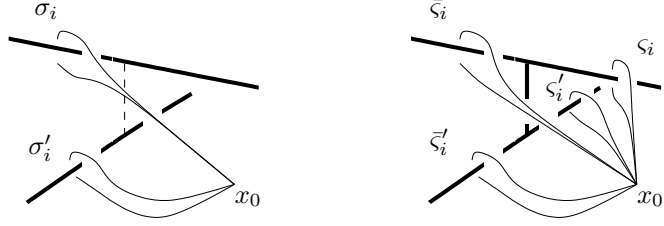}
}
\end{center}
   \caption{The loops for the meridians around the $i$-th tangle. 
When the singular $I$-fiber is smooth in $\mathcal O$
(on the left) and when it is in the branching locus
(on the right)}\label{fig:loopstangle}
\end{figure}

For a singular $I$-fiber,  the product $\sigma_i\sigma_i'$ 
($\varsigma_i\varsigma_i'$ when the fiber is in the branching locus)
projects in  $\pi_1(\mathcal O^3)$
to a root of $f^{\pm 1}$, but not in $\pi_1(M)$. On the other hand, if $e_i$ and
$e_{i+1}$ are the edges adjacent to
the $i$-th vertex, then
$$
m_i, m_i',m_{i+1}, m_{i+1}'\in  \pi_1(V_i\setminus (\Sigma_{\mathcal O}\cap V_i)).
$$

For an elliptic element $a\in\operatorname{Isom}^+(\mathbf H^3)$, let
$\Axis(a)\subset\mathbf H^3$ denote its fixed point set (or its axis). 

When the $I$-fiber is not in the branching locus,  
the angle between the axis of $\rho_0(\sigma_i)$ and $\rho_0(\sigma_i')$ is
\begin{equation}
 \label{equation:angle}
\angle(\Axis(\rho_0(\sigma_i)),\Axis(\rho_0(\sigma_i')))=\frac{p_i}{q_i}\pi
\end{equation}
with $p_i,q_i\in\mathbf Z$ coprime, $0<p_i<q_i$. 
This rational number
${p_i}/{q_i}$
describes the singularity of the fiber, that has order $q_i$. 
The  angle
of $P^2$ at the corresponding vertex is $\pi/q_i$.

When the $I$-fiber is in the branching locus,  
the angle between the axis of $\rho_0(\varsigma_i)$ and $\rho_0(\varsigma_i')$ is
\begin{equation}
 \label{equation:angle2}
\angle(\Axis(\rho_0(\varsigma_i)),\Axis(\rho_0(\varsigma_i')))=\frac{p_i}{2 q_i}\vartheta_i
\end{equation}
with $p_i,q_i\in\mathbf Z$ coprime, $0<p_i<q_i$ as above and $\vartheta_i=2\pi/m_i$ is the \emph{orbifold} angle,
where $m_i\geq 2$ is the order of the branching.
However, in what follows we may also consider any $0<\vartheta_i<2\pi$.
The  angle
of $P^2$ at the corresponding vertex is $\frac{\vartheta_i}{2q_i}$.

\begin{Definition}
\label{dfn:euclideanmodel}
The \emph{euclidean model} is the metric orbifold 
$$E(p_i/q_i)=\mathbf
R^3/D_{\infty},$$
 where $D_{\infty}$ is the infinite dihedral 
group generated by two rotations
of order $2$, whose axis are at distance one
and have an angle (after parallel transport) equal to $\frac{p_i}{q_i}\pi$.
\end{Definition}

\begin{Definition}
\label{dfn:singeuclideanmodel}
The \emph{singular euclidean model} is the cone manifold  
$$E(p_i/q_i,\vartheta_i)=\mathbf
R^3(\vartheta_i)/D_{\infty},$$
 where $\mathbf R^3(\vartheta_i)=\mathbf R^2(\vartheta_i)\times \mathbf R$ and
$\mathbf R^2(\vartheta_i)$ is the Euclidean plane with a singular point of angle $0<\vartheta_i<2\pi$. Here
$D_{\infty}$ is generated by two rotations
of order $2$, with axis at distance one perpendicular to the singular axis of $\mathbf R^3(\vartheta_i)$, and 
forming an angle (after parallel transport) equal to $\frac{p_i}{2q_i}\vartheta_i$.
\end{Definition}

\begin{Remark}
\label{remark:fibration}
The orbifold $E(p_i/q_i)$ and the cone manifold $E(p_i/q_i,\vartheta_i)$
have a natural fibration, that
gives precisely the fibration of a neighborhood of the i-th singular vertex.
This is the fibration by
parallel lines of $\mathbf R^3$, in the direction of the translation vector of the
index two subgroup
$\mathbf Z< D_{\infty}$. 
\end{Remark}

An alternative way of describing $E(p_i/q_i)$ is by considering fundamental
domains, cf.~Figure~\ref{fig:euclidean}. Consider a region of 
$\mathbf R^3$ bounded by two parallel planes at distance one. On each plane,
there is a rotation axis, one for each generator,
and $E(p_i/q_i)$ is obtained from identifying half of each face with the other half
after folding. For $E(p_i/q_i,\vartheta_i)$, a similar fundamental domain is constructed in 
$\mathbf R^3(\vartheta_i)=\mathbf R^2(\vartheta_i)\times\mathbf R$.

The fibers come from the vertical segments (say the planes are horizontal),  the
singular $I$-fiber is the minimizing segment between the rotation axis. It is 
its soul, in the Cheeger-Gromoll sense.

\begin{Definition}
A sequence of pointed metric spaces $(X_n,x_n)$ \emph{converges} to $(X_{\infty},x_{\infty})$ for the 
\emph{pointed bi-Lipschitz}
topology if, $\forall R>0$ and $\varepsilon>0$, there exists $n_0$ such that, for $n\geq n_0$,
$B(x_{\infty},R)$ is $(1+\varepsilon)$-bi-Lipschitz to a neighborhood $U\subset X_n$ that satisfies
$B(x_n,R-\varepsilon)\subseteq U\subseteq B(x_n,R+\varepsilon)$.
\end{Definition}

\begin{figure}
\begin{center}
{\psfrag{p}{${\pi}$}
\includegraphics[height=3cm]{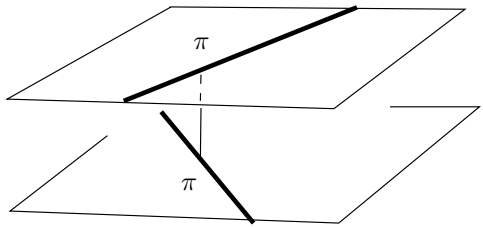}
}
\end{center}
   \caption{The model $E(p_i/q_i)$ via its fundamental domain.}
\label{fig:euclidean}
\end{figure}

When constructing the developing maps in Section~\ref{section:developingmaps},
we will use the following lemma for the transition between
singular and regular $I$-fibers:

\begin{Lemma}
 \label{lemma:euctrans}
Let $x_n$ be a sequence of points in the singular locus of $E(p_i/q_i)$. If $x_n\to
\infty$, then 
$(E(p_i/q_i),x_n)$ converges to another euclidean model with parallel singular
axis, for the pointed bi-Lipschitz topology. In addition, the distance between the axis is
 $q_i$, the order of the singular fiber.

The same statement holds true for the cone manifold $E(p_i/q_i,\vartheta_i)$ and points in the horizontal singular locus.
\end{Lemma}

\begin{proof} 
We prove it for $E(p_i/q_i)$, the proof for $E(p_i/q_i,\vartheta_i)$ being similar.
The lift  of the branching locus of $E(p_i/q_i)$ to the universal covering
(isometric to $\mathbf R^3$)
is a countable family of lines, all of them 
perpendicular to a given axis that minimize the distance between any pair of the lines. From
each line, we obtain the next one by a 
screw motion. This screw motion has axis the line perpendicular to all the lifts,
translation length one and rotation angle
$\frac{p_i}{q_i}\pi$. In this way, if $x_n$ goes to infinity along one of the
lines, the closest singular component will be 
parallel and at distance $q_i$. 
Then the convergence follows easily.
\end{proof}

\section{The Killing vector field}
\label{section:Killing}

In this section we prove a result  about Killing vector fields that will be used in the 
construction of developing maps.

Consider $P^2\setminus \Sigma_{P^2}$ the smooth part of $P^2$. Via the developing map of
the transversely hyperbolic foliation, 
the closure 
$$
\mathcal P= \overline{D_0(P^2\setminus \Sigma_{P^2})}
$$ 
is a polygon in $\mathbf
H^2\subset\mathbf H^3$. 
Let $m_i,m_i'\in\pi_1(M,x_0)$ be as in Section~\ref{section:fibration}, for $i=1,\ldots,n$.
Let $\tilde m_i$ and $\tilde m_i'$ be the corresponding paths lifted to the universal covering.
We may assume that the path
$D_0(\tilde m_i)$ starts at the base point 
$D_0(\tilde x_0)$ in the interior, crosses the boundary of $\mathcal P$,
and follows along $\rho_0(m_i)(\mathcal P)=D_0(m_iP)$ until $\rho_0(m_i)(D_0(\tilde x_0))$,
and similarly for $m_i'$, $\sigma_j$ and $\sigma_j'$.
Recall that $f=m_i(m_i')^{\pm 1}$, where the sign of the power 
$(m_i')^{\pm 1}$ depends on  the compatibility of orientations at a given axis.

By analyticity, if $\rho_t(f)$ is nontrivial, then  there is a natural way to associate a Killing vector field $F$ to 
the deformation of $\rho_t(f)$. Namely, as $\rho_0(f)=\pm \operatorname{Id}$,
$$
\rho_t(f)= \pm \exp(t^s \mathfrak f+ O(t^{s+1}))
$$
for some $\mathfrak f\in\mathfrak{sl}_2(\mathbf C)$. The Killing vector field $F$ associated to the infinitesimal isometry
 $ \mathfrak f$ is then
$$
F_x=\lim_{t\to 0} \frac {\rho_t(f)(x)-x}{t^s}\qquad \forall x\in \mathbf H^3.
$$

The \emph{incenter} of a polygon is the point whose distance to \emph{every} edge of the polygon is the same, if it exists.

The goal of this section is to prove the following:

\begin{Proposition}
 \label{lemma:killingfield}
The polygon $\mathcal P$ has an incenter and the Killing vector field $F$ 
 is a field of infinitesimal purely loxodromic translations along an axis that   meets perpendicularly $\mathcal P$ in its incenter.
In particular $F$ is perpendicular to $\mathcal P$.
In addition,  it has the same orientation as the fiber of the Seifert fibration of $\mathcal O^3$
restricted to the interior of $\mathcal P$.
\end{Proposition}

Notice that  the interior of $\mathcal P$ is orientable because the mirror points are in $\partial\mathcal P$, thus it makes sense to talk
about the induced orientation of the fiber in $\mathcal O^3$ and of the Killing field on $\mathbf H^3$.

Before proving the proposition, we need to show that $\rho_t(f)$ is nontrivial.

For a representation $\rho_t$, a
pseudo\-developing map is a $\rho_t$-equivariant
map $D_t\!:\tilde M\to \mathbf H^3$, 
such that around the singular locus it is like the developing map around a cone
singularity (ie.\ conical in a tubular neighborhood). 
This $D_t$ can be used to define a volume of $\rho_t$ \cite{Stefano}.

\begin{Lemma}
\label{lemma:schlafli}
For $\rho_t$ satisfying Proposition~\ref{assumption:rho}, there 
 exists a uniform constant $C>0$ such that, for $t>0$ close to $0$:
$$
\operatorname{Vol}(\rho_t)\geq C t^r.
$$
\end{Lemma}

\begin{proof}
 Schl\"afli's formula applied to cone manifolds \cite{PortiEuclidean} gives:  
$$
\operatorname{Vol}(\rho_t)=-\frac12\int_0^t\sum_{e}\operatorname{length}(e) d\alpha_e,
$$
where the sum runs over all singular edges or components. In our case, as the length is bounded below, and the cone angles 
are $\pi-t^r+ O(t^{r+1})$, the lemma is straightforward.
\end{proof}

\begin{Lemma}
\label{lemma:nontrivial}
 For small $t>0$, $\rho_t(f)$ is nontrivial.
\end{Lemma}

\begin{proof}
Seeking a contradiction, assume that $\rho_t(f)$ is trivial. Then,
$\rho_t(m_i')=\rho_t(m_i)^{\pm 1}$.

 We first claim that 
for small values of $t>0$, $\rho_t$ of the $i$-th tangle group is elliptic, ie.\
for a tubular neighborhood $V_i$ of the i-th $I$-fiber,
$\rho_t(\pi_1(V_i\setminus (V_i\cap \Sigma_{\mathcal O})))$ is elliptic.
We assume first that  the singular $I$-fiber is not in the branching locus of the orbifold.
Again by contradiction, assume that 
the axis of
 $\rho_t(\sigma_i)$ and $\rho_t(\sigma_i')$
are disjoint.
Then there is a minimizing segment between the
axis of $\rho_t(\sigma_i)$ and $\rho_t(\sigma_i')$, because
  the axis of $\rho_0(\sigma_i)$ and $\rho_0(\sigma_i')$ meet at one
point  with angle $\frac{\pi}{q_i}$.
 By rescaling the hyperbolic space in such a way that the length
of this segment is one, and by taking the pointed limit with base point the 
midpoint of this segment,
we look at the limits of the axis $\Axis(\rho_t(\sigma_i))$ and $
\Axis(\rho_t(\sigma_i'))$ after
rescaling: we obtain two euclidean lines
at distance one and forming an angle, as in the euclidean model of
Definition~\ref{dfn:euclideanmodel}.
In this model, the axis of the $m_i$ and $m_i'$ are parallel but
different, by Lemma~\ref{lemma:euctrans}.
 This contradicts that 
 $\rho_t(m_i')=\rho_t(m_i)^{\pm 1}$, and hence $\Axis(\rho_t(\sigma_i))$ and 
$\Axis(\rho_t(\sigma_i'))$ meet at one point.
When the singular $I$-fiber is in the branching locus of the orbifold, then a similar argument tells that 
the segment between 
$\Axis(\rho_t(\varsigma_i))\cap \Axis(\rho_t(\bar \varsigma_i))$ and 
$\Axis(\rho_t(\varsigma_i'))\cap \Axis(\rho_t(\bar \varsigma_i'))$ has length zero.

Construct a pseudodeveloping map $D_t\!:\tilde M\to \mathbf H^3$ as follows. Start by mapping a tubular neighborhood of the 
singularity to a tubular neighborhood of the axis of the corresponding elements via $\rho_t$.
Now, since $\rho_t$ of the edge groups is elliptic,
the singular $I$-fiber  can be mapped to a neighborhood of this point. Similarly, as
$\Axis(\rho_t(m_i))=\Axis(\rho_t(m_i'))$, the regular $I$-fibers can be mapped to a $\delta$-neighborhood of the
axis, for $\delta>0$ arbitrarily small.
The boundary of the neighborhood of the $I$-fibers  is a torus, and since 
$\rho_t(f)$ is trivial, this torus can be deformed $\rho_t$-equivariantly to a
circle, in a neighborhood of radius $2\delta$. Extend $D_t$ by collapsing the
rest of the
 manifold to a disk. Thus $\rho(t)$ has arbitrary small volume, by choosing the
$\delta>0$ small enough, 
contradicting Lemma~\ref{lemma:schlafli}.
\end{proof}

Recall that the Killing vector field $F$ is the corresponding
field of the infinitesimal isometry $\mathfrak {f}\in\mathfrak{sl}_2(\mathbf C)$, where 
$\rho_t(f)=\pm\exp( t^s \mathfrak {f}+ O( t^{s+1}))$. 
By Lemma~\ref{lemma:nontrivial}, $\mathfrak {f}\neq 0$ and $s$ is well defined.

\begin{Lemma}
\label{lemma:sleqr} If $\rho_t(f)=\exp( t^s \mathfrak {f}+ O( t^{s+1}))$, and $r\in\mathbf N$ is as in
Proposition~\ref{assumption:rho},
then 
$$
s\leq r.
$$
\end{Lemma}

\begin{proof}By Lemma~\ref{lemma:schlafli},
$\operatorname{Vol}(\rho_t)\geq C t^r$ for some uniform constant $C>0$.

On the other hand, the displacement function
of $\rho_t(f)$ in a compact neighborhood $U$ of $\mathcal P$ is $\leq C_0 t^s$. In particular,
the Hausdorff distance between $\Axis(\rho_t(m_i))\cap U$ and $\Axis(\rho_t(m_i'))\cap U$ is
$\leq C_1 t^s$. 

We want to construct a pseudodeveloping map with volume $\leq C' t^s$.
We start by constructing a developing map  around the singular locus, by taking a small 
radius of the tube, with arbitrarily small volume, say $\leq t^s$. 
Moreover as the Hausdorff distance between $\Axis(\rho_t(m_i))\cap U$ and $\Axis(\rho_t(m_i'))\cap U$ is
$\leq C_1 t^s$, we can develop a solid torus that is a neighborhood of the $I$-fibers with volume $\leq C_ 2 t^s$,
and so that the length of the fiber is $\leq 3 C_1 t^s$. The exterior of this torus in $\mathcal O^3$
is a solid torus without singularity $V$, and since the displacement function  of $\rho_t(f)$ in  $U$ is $\leq C_0 t^s$,
the pseudodeveloping map can be extended to $V$ with a volume contribution $\leq C_3 t^s$.
Thus, the volume of the pseudodeveloping map, and of $\rho_t$, is $\leq C' t^s$. Comparing both 
inequalities for the volume:
$$
C t^r\leq\operatorname{Vol}(\rho_t)\leq C' t^s,
$$
for small values of $t>0$. Thus $s \leq r$.
\end{proof}

Before proving Proposition~\ref{lemma:killingfield}, we still need a further computation.
Let $$
B:\mathfrak{sl}_2(\mathbf C)\times \mathfrak{sl}_2(\mathbf C)\to\mathbf C
$$ 
denote the complex Killing form, see Appendix~\ref{sec:inf}.
For $\mathfrak{a},\mathfrak{b}\in \mathfrak{sl}_2(\mathbf C)$, 
$$
B(\mathfrak{a},\mathfrak{b})=
\operatorname{Trace}(Ad_{\mathfrak{a}}\circ Ad_{\mathfrak{b}})= 
4 \operatorname{Trace}(\mathfrak{a}\mathfrak{b}).
$$

\begin{Definition}
\label{definition:complexlength} 
We say that an infinitesimal isometry $\mathfrak{a}\in\mathfrak{sl}_2(\mathbf C)$ has
\emph{complex length} $l\in \mathbf C$ if $\exp(t\mathfrak{a})$ has complex length $t\,l$.
\end{Definition}

\begin{Lemma}
\label{lemma:productofF}
Let $\mathfrak d_i\in\mathfrak{sl}_2(\mathbf C)$ denote infinitesimal rotation
of complex length $\pi \mathbf i$ around the $i$-the oriented axis of $\mathcal P$.
Then
$$
B(\mathfrak d_{i},\mathfrak f)=\left\{ 
\begin{array}{l}
0, \textrm{ if } s<r. \\
4, \textrm{ if } s=r.
\end{array}
     \right.
$$
In particular $B(\mathfrak d_{i},\mathfrak f)$ is independent of $i$.
\end{Lemma}

We will show later in Lemma~\ref{rem:r=s} that only the case  $B(\mathfrak d_{i},\mathfrak f)=4$ occurs, in particular $s=r$.

\begin{proof}
We can find  $\Lambda_t, \Lambda' _t\in SL_2(\mathbf C)$ that  depend analytically on $t^{1/2}$,
 $t\in (0,\varepsilon)$ so that $\Lambda_0=\Lambda_0'=\operatorname{Id}$, 
$$
\Axis(\rho_t(m_i))= \Lambda_t(\Axis (\rho_0(m_i)))
\qquad\textrm{ and }
\qquad
\Axis(\rho_t(m_i'))= \Lambda_t' (\Axis (\rho_0(m_i'))).
$$
Those matrices $\Lambda_t$ and $\Lambda_t'$ are obtained by solving the characteristic polynomials for
$\rho_t(m_i)$ and $\rho_t(m_i')$, hence they are analytic on $t^{1/2}$.

Assume that the axis of 
$\rho_0(m_i)$ is $\overline{0\infty}$ in the upper half space model of $\mathbf H^3$. If 
 the orientations of $m_i$ and $m_i'$ are compatible, then.
$$
\rho_t(m_i)=\pm \Lambda_t\begin{pmatrix}
                e^{i\alpha_i/2} & 0 \\
		0 & e^{-i\alpha_i/2}
               \end{pmatrix}
\Lambda_t^{-1},
\quad
\rho_t(m_i')=\pm \Lambda_t' \begin{pmatrix}
                e^{i\alpha_i'/2} & 0 \\
		0 & e^{-i\alpha_i'/2}
               \end{pmatrix}
 (\Lambda_t')^{-1}.
$$
Notice that since $\rho_0(f)=\pm\operatorname{Id}$, 
 if $\rho_t(f)=\pm \exp(t^s\mathfrak{f}+ O(t^{s+1}))$,
then 
$$\Lambda_t^{-1}\rho_t(f)\Lambda_t=\pm \exp(t^s\mathfrak{f}+ O(t^{s+1})),$$
 hence we may assume 
$\Lambda_t=\operatorname{Id}$ (after replacing $\Lambda'_t$ by $\Lambda_t^{-1}\Lambda'_t$).
Let
$$\Lambda_t' =\begin{pmatrix}
                                                   1+ a t^{\nu} & b t^{\nu} \\
						   c t^{\nu} & 1-a t^{\nu}
                                                  \end{pmatrix} 
+ O(t^{\nu+1/2}),
$$
with $a,b,c\in\mathbf C$, $\nu\in\frac12\mathbf N$.
Since $\alpha_i(t)= \pi-t^r+ O(t^{r+1})$ and $\alpha_i'(t)= \pi-t^r+ O(t^{r+1})$,
\begin{equation}
\label{eqn:rhotfa} 
\rho_t(f)= \rho_t(m_i)\rho_t(m_i')=\pm
	\begin{pmatrix}
	  -1+\mathbf i t^r & 2 t^{\nu} b \\
	2 t^{\nu} c &  -1-\mathbf i t^r 
	\end{pmatrix}
+ O( t^{\min(r,\nu)+1/2}).
\end{equation}
When $m_i$ and $m_i'$ have opposite orientation, then 
$$
\rho_t(m_i')=\pm \Lambda_t' \begin{pmatrix}
                e^{-i\alpha_i'/2} & 0 \\
		0 & e^{i\alpha_i'/2}
               \end{pmatrix}
 (\Lambda_t')^{-1} ,
$$
and since $f=m_i (m_i')^{-1}$, (\ref{eqn:rhotfa}) also holds true.

Let $\mathfrak{d}_i$ be the infinitesimal rotation around the oriented axis of 
$\rho_0(m_i)$ of complex length $\pi \mathbf i$. In this model:
$$
\mathfrak d_i
=\begin{pmatrix} \mathbf i/2 & 0 \\ 0 & -\mathbf i/2
                               \end{pmatrix}.
$$
From (\ref{eqn:rhotfa}) we distinguish two cases:
\begin{itemize}
\item[1)] If $\nu<r$, then $s=\nu<r$ and 
\begin{equation}
\label{eqn:f0} 
\mathfrak{f}=\begin{pmatrix} 0 & -2 b \\ -2c & 0 
                               \end{pmatrix}.
\end{equation}
\item[2)] If $\nu\geq r$, then $s=r$ and 
\begin{equation}
\label{eqn:fi} 
\mathfrak{f}=\begin{pmatrix} -\mathbf i & -2 b \\ -2c & \mathbf i
                               \end{pmatrix}.
\end{equation}

This includes the case $\nu> r$, with $b=c=0$.
\end{itemize}

Then the formula follows from 
$B(\mathfrak d_{i},\mathfrak f)=4\operatorname{Trace}(B(\mathfrak d_{i}\mathfrak f))$.
\end{proof}

\begin{Remark}
\label{rem:Fperpendicular}
It follows from the proof of Lemma~\ref{lemma:productofF} that the Killing vector field $F$  is perpendicular to the axis  $\Axis(\rho_0(m_i))$. 
This holds from Equalities (\ref{eqn:f0}) and (\ref{eqn:fi}), because in both cases the real part of the diagonal of $\mathfrak f$ vanishes, and the axis 
is $\Axis(\rho_0(m_i))=\overline{0\infty}$.
\end{Remark}

\begin{Lemma}
\label{rem:r=s}
 $r=s$.
\end{Lemma}

\begin{proof}
 Assume  that $s<r$, hence $B(\mathfrak f, \mathfrak d_i)=0$ for each $i=1,\ldots,n$. 
Using the formulas of Appendix~\ref{sec:inf}, we shall find a contradiction. 
When $\mathfrak f$ is non parabolic,
let $\Axis(\mathfrak f)\subset\mathbf H^3$ denote the axis of $\mathfrak f$, 
which is the minimizing set for the norm of the Killing vector field $\vert F\vert$.
If $f$ is non parabolic, then by Proposition~\ref{prop:killing} the complex distance
between $\Axis(\mathfrak f)$ and $\Axis(\mathfrak d_i)$ is $\pm\frac{\pi}2\mathbf i$, hence 
$\Axis(\mathfrak f)$ must meet perpendicularly  all edges of $\mathcal P$, which is impossible.
So we assume that $\mathfrak f$ is  parabolic. In this case, Proposition~\ref{prop:killing1P} tells that
the point at $\infty$ fixed by $\mathfrak f$ is an endpoint of all (infinite) edges of $\mathcal P$, which is again impossible.
\end{proof}

\begin{proof}[Proof of Proposition~\ref{lemma:killingfield}]
By Lemmas~\ref{rem:r=s} and \ref{lemma:productofF}, $r=s$ and $B(\mathfrak f, \mathfrak d_i)=4$ for each $i=1,\ldots,n$. 
We discuss again the possibilities for $\mathfrak f$. 
If $\mathfrak f$ was parabolic, then Corollary~\ref{cor:equihorsphere} would tell that all (infinite) edges
of $\mathcal P$ are tangent to a given horosphere, and that their tangent vectors are parallel in this horosphere, which is again impossible.
Hence we are left with the case that $\mathfrak f$ is nonparabolic and has an axis 
whose complex distance to all oriented edges of $\mathcal P$ is the same
(by Proposition~\ref{prop:killing}).

Notice that by Remark~\ref{rem:Fperpendicular}, the Killing 
vector field $ F$
is perpendicular to every edge of $\mathcal P$. 
Hence, at the vertices of $\mathcal P$, $F$ is perpendicular to the plane containing $\mathcal P$,
and since it is a Killing vector field, $F$ is perpendicular to $\mathcal P$.
Thus $\mathfrak f$ is either an infinitesimal rotation with axis coplanar to $\mathcal P$ or
an infinitesimal translation with axis perpendicular to $\mathcal P$.
If $\mathfrak f$ is an infinitesimal rotation  
then by Remark~\ref{remark:killing} (Equation \ref{eqn:killingnp}) the complex distance between $\Axis(\mathfrak f)$ 
and every oriented axis of $\mathcal P$ is the same, but this is impossible in a coplanar configuration.
Thus $\mathfrak f$ is an infinitesimal translation, and its axis meets $\mathcal P$
perpendicularly and is equidistant to all edges of $\mathcal P$.

Finally, the assertion about orientations follows from the next lemma.
\end{proof}

\begin{Lemma}
If the cone angles decrease, then the orientation of the Killing vector  field
$F$ is the same as the orientation of the fiber in $\mathcal O^3$. 
If they increase, then it is the opposite orientation.
\end{Lemma}

\begin{proof}
By Lemma~\ref{lemma:schlafli}, the volume of the representation is positive,
$
\operatorname{Vol}(\rho_t)>0$ for $t>0$.
On the other hand, if the orientation of the Killing vector field was the wrong one,
we would be able to construct a pseudodeveloping map with negative volume, following the strategy of
Lemma~\ref{lemma:nontrivial}.
\end{proof}

\begin{Corollary}
\label{coro:rhotfhyp} 
For $t\in (0,\varepsilon)$, $\rho_t(f)$ is loxodromic (ie.\ not elliptic nor parabolic).
\end{Corollary}

\begin{proof}
Assume first that $\rho_0(f)=\operatorname{Id}$.
 Since $\mathfrak f$ is hyperbolic, then the first nonzero derivative of the trace of $\rho_t(f)$ is
real positive, in particular for small values of $t>0$ it is not
contained in $[-2,2]$. A similar argument applies when $\rho_0(f)=-\operatorname{Id}$.
\end{proof}

\section{Constructing developing maps}
\label{section:developingmaps}

Along this section, assume that $\rho_t$, $t\in [0,\varepsilon)$, is a path of representations that 
satisfies the conclusion of Proposition~\ref{assumption:rho}.
The goal is to construct developing maps with holonomy $\rho_t$.

We construct the developing maps in three steps. Firstly, in a neighborhood of
the vertices of $P^2$, that correspond to tangles of the orbifold, or singular
interval fibers.
Secondly, on the edges, and finally on the interior.

We start with the vertices of $P^2$, ie.\ the tangles of $\mathcal O^3$.

We assume for the moment that the $I$-fiber of the  $i$-th vertex  is not in the branching locus of the orbifold. 
(See Remark~\ref{Remark:singularcore} when it is in the branching locus of the orbifold).
Let $\sigma_i$ and $\sigma_i'$ in $\pi_1(M)$ denote
the meridians corresponding to the $i$-th tangle, as in
Section~\ref{section:fibration} (Figure~\ref{fig:loops}).

\begin{Lemma}
\label{lemma:positivedistance} 
For $t>0$, 
$$
\dist(\Axis(\rho_t(\sigma_i)),\Axis(\rho_t(\sigma_i'))>0.
$$
Moreover, there is a shortest segment $\nu_i(t)$ between both axis that
converges to the $i$-th vertex of the polygon as $t\to 0^+.$
\end{Lemma}

\begin{proof}
By contradiction, 
assume that  $\dist(\Axis(\rho_t(\sigma_i)),\Axis(\rho_t(\sigma_i')))=0$ for small values
of $t>0$.
 Since $\angle(\Axis(\rho_0(\sigma_i)),\Axis(\rho_0(\sigma_i'))=  \pi p_i/q_i$, 
  $\langle \rho_t(\sigma_i),\rho_t(\sigma_i')\rangle$ is an elliptic group
that fixes 
a point close to the initial vertex in $\mathbf H^3$. In particular, since
$m_i,m_i'\in \langle\sigma_i,\sigma_i'\rangle $,
 $\rho_t(f^{\pm 1}) =\rho_t(m_i)\, \rho_t(m_i')^{\pm 1}$ is either trivial or elliptic,
which contradicts
Corollary~\ref{coro:rhotfhyp}.

The existence of the shortest segment  $\nu_i(t)$ comes from the fact that
$\Axis(\rho_0(\sigma_i))$ and $\Axis(\rho_0(\sigma_i'))$ meet at one point with
angle $\pi \frac{p_i}{q_i}$, so the distance function between both axis is a proper
convex function on $\Axis(\rho_0(\sigma_i))\times  \Axis(\rho_0(\sigma_i'))$
 and has a minimum. Therefore, for small $t>0$ it is also a proper convex function 
on $\Axis(\rho_t(\sigma_i))\times  \Axis(\rho_t(\sigma_i'))$
and 
has a minimum.
\end{proof}

The idea now is to construct a double roof ${\mathcal R}_i(t)$ around $\nu_i(t)$ as
follows. Consider an embedding of both axis $\Axis(\rho_t(\sigma_i))$ and 
$\Axis(\rho_t(\sigma_i'))$ and the common perpendicular $\nu_i(t)$ in
$\mathbf H^3$.
Now consider two sectors, one with axis $\Axis(\rho_t(\sigma_i))$ and angle 
$\alpha_i(t)$, another one with axis $\Axis(\rho_t(\sigma_i'))$ and angle 
$\alpha_i'(t)$. (Here 
$\alpha_i(t)$ and $\alpha_i'(t)$ are
the respective rotation angles  of $\rho_t(\sigma_i)$ and $\rho_t(\sigma_i')$).
Choose the sectors so that $\nu_i(t)$ is bisector to both of them, 
and consider the intersection (Figure~\ref{fig:roofs}).

\begin{figure}
\begin{center}
{\psfrag{n}{$\nu_i(t)$}
\psfrag{s}{$\Axis(\rho_t(\sigma_i))$}
\psfrag{s'}{$\Axis(\rho_t(\sigma_i'))$}
\includegraphics[height=4cm]{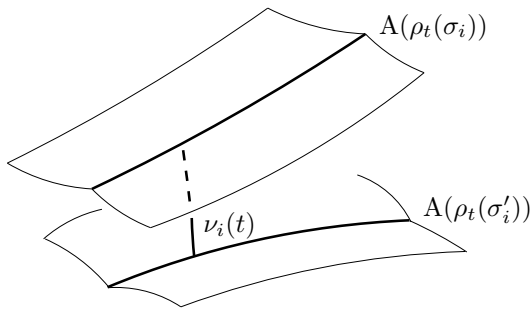}
}
\end{center}
   \caption{The double roof. The tubular neighborhood here is 
${\mathcal R}_i(t)$}\label{fig:roofs}
\end{figure}

The boundary of these sectors may intersect.
Let $r_i(t)>0$ be the maximal radius such that the tubular neighborhood
$\mathcal N_{r_i(t)}(\nu_i(t))$
 does not meet the intersection of  the sides of the sectors.
We denote ${\mathcal R}_i(t)=\mathcal N_{r_i(t)/2}(\nu_i(t))$ the tubular neighborhood of
$\nu_i(t)$ in this double roof. Notice that possibly $r_i(t)\to 0$ as $t\to
0^+$,
 but:

\begin{Lemma}
\label{lemma:r/nu}
$$
\lim_{t\to 0^+}  \frac{r_i(t)}{\vert\nu_i(t)\vert}=+\infty.
$$ 
\end{Lemma}

\begin{proof}
We cut the double roof along the hyperplane perpendicular to $\nu_i(t)$ that
contains its midpoint, 
and consider each roof separately. We bound below the distance from $\nu_i(t)$ to
the intersection of 
each piece of the roof to this hyperplane, and it suffices to discuss the
argument for one of the 
edges, say $\sigma_i$.
Let $\alpha_i(t)$ denote the cone angle, which is the angle of the roof. 
By comparison with the euclidean right triangle (Figure~\ref{fig:triangle}):
$$ 
\frac{r_i(t)}{\vert\nu_i(t)\vert/2}\geq \tan\frac{\alpha_i(t)}2\to \infty
\quad\textrm{ as } t\to 0^+,
$$
because $\alpha_i(0)=\pi$.
\end{proof}

\begin{figure}
\begin{center}
{\psfrag{n}{${\vert\nu_i(t)\vert}/2$}
\psfrag{a}{${\alpha_i(t)}/2$}
\psfrag{r}{$ r_i(t)$}
\includegraphics[height=3cm]{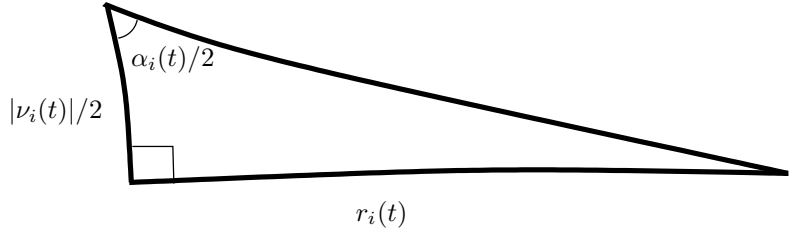}
}
\end{center}
   \caption{The hyperbolic triangle approximated by a euclidean one.}
\label{fig:triangle}
\end{figure}

Let $x_i(t)$ denote the midpoint of $\nu_i(t)$. Let $\overline {\mathcal R}_i(t)$ be the
result of 
identifying the sides of each roof of ${\mathcal R}_i(t)$ by
a rotation around its edge, so that the edges become interior points.

From   Lemma~\ref{lemma:r/nu}, we get:

\begin{Corollary}
\label{cor:bilip}
 For the pointed bi-Lipschitz topology:
$$
\lim_{t\to 0^+}  \frac{1}{\vert\nu_i(t)\vert} (\overline {\mathcal R}_i(t), x_i(t))
=(E(p_i/q_i),x_{\infty}).
$$ 
\end{Corollary}

Next corollary deals with points of $\overline {\mathcal R}_i(t)$ away from the center.

\begin{Corollary}
\label{cor:R}
There exist $R_0>0$ and $t_0>0$ such that for 
  $0<t\leq t_0$ and  $x\in \overline {\mathcal R}_i(t)$
that it is singular and $R_0 \vert \nu_i(t)\vert \leq d(x,x_i(t))  < \frac12 r_i(t) $,
the following hold.
Let ${{\delta}}(x)$ be the distance between $x$ and the other singular component.
 Then the rescaled ball
$$
\frac 1{{{\delta}}(x)}B(x,10 {{\delta}}(x)) 
$$
is $3/2$-bi-Lipschitz to the corresponding ball in $E(0)$.
\end{Corollary}

\begin{proof}
By Corollary~\ref{cor:bilip}, it is sufficient to prove it for the euclidean models $E(p_i/q_i)$.
Then the corollary follows from Lemma~\ref{lemma:euctrans}.
\end{proof}

\begin{Remark}
\label{Remark:singularcore}
When the $I$-fiber of the $i$-th vertex is in the branching locus of the orbifold, then one needs to consider
the double roofs $\mathcal R_i(t)$ and 
the corresponding neighborhoods 
$\overline{\mathcal R_i(t)}$ with a singular core $\nu_i(t)$ of cone angle $\vartheta_i$.
Lemma~\ref{lemma:r/nu} and Corollaries~\ref{cor:bilip} and \ref{cor:R} apply in this case.
\end{Remark}

Next we deal with the edges of $P^2$. 
We shall construct locally the hyperbolic structures in pieces
 $\overline{\mathcal S}(q)$ and study its behavior and compatibility in 
Corollary~\ref{cor:bilipedge} and Lemma~\ref{lemma:glueS}. 
Before that, we need few technical results about the 
edges
 $\Axis(\rho_t(m_i))$ and 
$\Axis(\rho_t(m_i'))$.

To simplify notation, set
$i=1$. The endpoints of the segment  $e_1$ of $\mathcal P^2$ at time $t=0$ 
are  $v_1$
and $v_2$. But for  $t>0$, we consider two segments $e_1(t)$ and
$e_1'(t)$ that are contained in  $\Axis(\rho_t(m_1))$ and 
$\Axis(\rho_t(m_1'))$, respectively, and whose endpoints are given by the
$\sigma$'s or $\varsigma$'s: ie.\ the endpoints of  
the corresponding conjugates of 
$\nu_1(t)$ and $\nu_2(t)$.

Let $p_1(t)$ and $p_2(t)$ denote the endpoints of $e_1(t)$.
For $q\in e_1(t)$, let $q'\in \Axis(\rho_t(m_1'))$ be the point that realizes the distance 
between $q$ and $\Axis(\rho_t(m_1'))$ (cf.\ Fig.~\ref{fig:dq}). Define, for $q\in e_1$:
$$
	{{\delta}}_t(q)=d(q,q')=d(q,\Axis(\rho_t(m_1'))).
$$

\begin{Lemma}
\label{Lemma:uniformedges} Let $q\in e_1(t)$ and $q'\in\Axis(\rho_t(m'_1))$ be as above.
\begin{enumerate}
 \item The distance ${{\delta}}_t(q)=d(q,q')$ converges to zero uniformly on $q\in e_1(t)$:
 $$\lim_{t\to 0^+} \sup_{q\in e_1(t)} {{\delta}}_t(q)=0$$
\item Let $v_{q,t}\in T_q\mathbf H^3$ be the parallel transport of the tangent vector to $e_1'(t)$
 along the segment $\overline{q'q}$. Then
$$
\lim_{t\to 0^+} \sup_{q\in e_1(t)} \angle_q e_1(t) v_{q,t}=0
$$
\item Let $R_0>0$ be as in Corollary~\ref{cor:R}. There exists $t_0>0$ such that, for $0<t<t_0$, $q\in   e_1(t)$ satisfies $d(q,p_1(t))>  R_0\,\vert\nu_1(t)\vert$ and $d(q,p_2(t))> R_0\,\vert\nu_2(t)\vert$, then:
$$
q'\in e_1'(t).
$$
\end{enumerate}
\end{Lemma}

\begin{proof}
By convexity of the distance function in hyperbolic space, we have, for $q\in e_1(t)$:
\begin{equation}
\label{eqn:dconvex}
{{\delta}}_t(q)\leq \max\{{{\delta}}_t(p_1(t)),{{\delta}}_t(p_2(t))\},
 \end{equation}
because $p_1(t)$ and $p_2(t)$ are the endpoints of $ e_1(t)$. This proves Assertion 1 of the lemma.

\begin{figure}
\begin{center}
{
\psfrag{e1}{$e_1(t)$}
\psfrag{e11}{$e_1'(t)$}
\psfrag{p1}{$p_1(t)$}
\psfrag{pn}{$p_2(t)$}
\psfrag{q}{$q$}
\psfrag{qq}{$q'$}
\includegraphics[height=2.1cm]{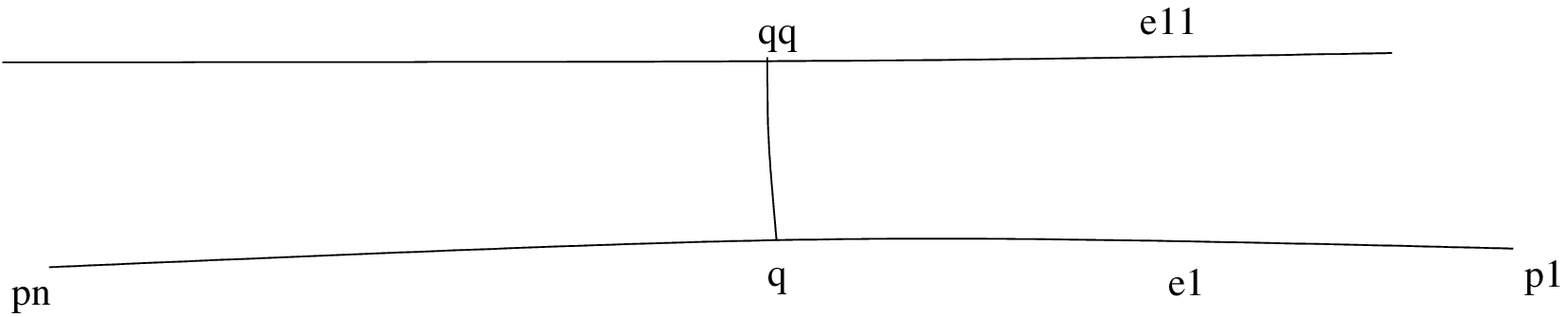}
}
\end{center}
   \caption{}\label{fig:dq}
\end{figure}

In order to prove Assertion 3, 
if $ d(q,p_1(t))=   R_0\,\vert\nu_1(t)\vert$ 
 or  
if $ d(q,p_2(t))=  R_0\,\vert\nu_2(t)\vert$,
then the assertion holds true for these $q$,
because of  Corollary~\ref{cor:R}. As those $q$ are extremal, for other $q$ the assertion follows from
Equation~(\ref{eqn:dconvex}) and elementary  arguments.

 Next we prove Assertion 2. 
Up to permuting $p_1$ with $p_2$, we may assume that $d(q, p_1(t))>\frac13\vert e_1(0)\vert$, where 
$\vert e_1(0)\vert$ denotes the length of $e_1(0)$.
Let $\beta_q(t)$ be the angle between $v_{q,t}$ and $e_1(t)$.
By the  triangle inequality in spherical space, the angle $\beta_q(t)$ satisfies:
$0\leq \beta_q(t)\leq  \beta_1 +\beta_2$, where
$\beta_1$ is the angle between $v_{q,t}$ and $qp_1'$,
  $\beta_2$ is the angle between $qp_1'$ and $qp_1\subset e_1(t)  \subset \Axis(\rho_t(m_1))$,
and $p_1'\in \Axis(\rho_t(m_1'))$ realizes $d(p_1,\Axis(\rho_t(m_1')) )=d(p_1,p_1')$,
cf.\ Figure~\ref{fig:beta}.

\begin{figure}
\begin{center}
{
\psfrag{v}{$v_{q,t}$}
\psfrag{b1}{$\beta_1$}
\psfrag{b2}{$\beta_2$}
\psfrag{b3}{$\beta_3$}
\psfrag{p}{$p_1'$}
\psfrag{pp}{$p''$}
\psfrag{q}{$q$}
\psfrag{qq}{$q'$}
\includegraphics[height=2.5cm]{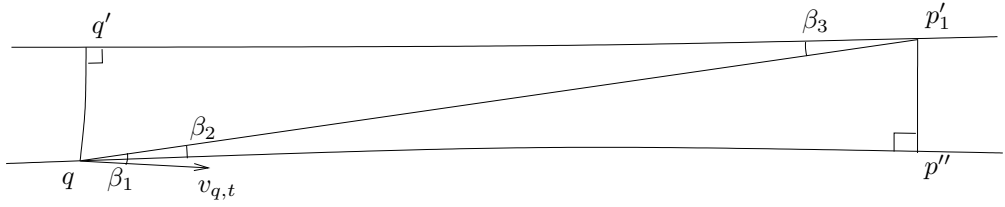}
}
\end{center}
   \caption{Triangles in the proof of Assertion 2 of Lemma~\ref{Lemma:uniformedges}}\label{fig:beta}
\end{figure}

Let $p''\in \Axis(\rho_t(m_1))$ realize the distance from $p_1'$ to $\Axis(\rho_t(m_1))$, so that 
$q$, $p_1'$ and $p''$ form a triangle with angles $\beta_2$ at $q$, and $\pi/2$ at $p''$. Then
$$
\tan\beta_2= \frac{\tanh d(p_1',p'')}{\sinh d(q,p'')}\leq  \frac{\tanh d(p_1',p'')}{\sinh(\frac13 \vert e_1(0)\vert- d(p_1,p'') )}
$$
which converges to zero uniformly on $q$.
Consider now the triangle $q$, $q'$ and $p_1'$. By the same argument as before the angle $\beta_3$ of this triangle at $p_1'$ converges to zero.
The angles of the triangle satisfy:
$$
	(\frac{\pi}2-\beta_1)+\frac{\pi}2+\beta_3=\pi-\textrm{Area}(qq'p_1').
$$
In addition, the area of this triangle converges to zero uniformly on $q$, by Assertion~1 of the lemma.
Thus
$$
	\beta_1 = \beta_3+\textrm{Area}(qq'p_1')\to 0,\qquad \textrm{ uniformly on } q .
$$
\end{proof}

We define, for  $0<t<t_0$ as in Assertion 3 of Lemma~\ref{Lemma:uniformedges}:
$$
	\hat e_1(t)= 
	\left\{q\in e_1(t)\mid d(q,p_1(t)) \geq  R_0\,\vert\nu_1(t)\vert \textrm{ and } 
	d(q,p_2(t))\geq  R_0\,\vert\nu_2(t)\vert\right\}.
$$

Using also Lemma~\ref{Lemma:uniformedges}, construct a double roof from the segment between $q$ and $q'$, with edges determined by  
$\Axis(\rho_t(m_1))$ and $\Axis(\rho_t(m_1'))$, and with dihedral angles the respective rotation angles of $\rho_t(m_1)$ and 
$\rho_t(m_1')$,
$\alpha_1(t)$ and
 $\alpha_1'(t)$, as before.
Let ${\mathcal S}(q,t)=B(q,s(q,t)/2)$  be the ball in this double roof, with $s(q,t)$ maximal such that the 
sides of the roof do not meet. As in Lemma~\ref{lemma:r/nu}:

\begin{Lemma}
 \label{lemma:s/nu}
$$
\lim_{t\to 0^+}\frac{s(q,t)}{d(q,q')}=+\infty
$$
uniformly on $q\in \hat e_1(t)$.
\end{Lemma}

The proof  of this limit is the same as Lemma~\ref{lemma:r/nu}, using the uniform limits of Lemma~\ref{Lemma:uniformedges}.

Identifying the sides  of ${\mathcal S}(q,t)$ by the rotations corresponding to its edges, we obtain $\overline {\mathcal S}(q,t)$. From Lemmas~\ref{lemma:s/nu} and 
\ref{Lemma:uniformedges}, we get:

\begin{Corollary}
\label{cor:bilipedge}
For any choice of $q(t)\in\hat e_1(t)$ and  for the pointed bi-Lipschitz topology:
$$
\lim_{t\to 0^+}  \frac{1}{{{\delta}}_t(q)} (\overline{\mathcal S}(q,t), q(t))
=(E(0),q_{\infty}),
$$ 
uniformly on $q$.
\end{Corollary}

Recall that  ${{\delta}}_t(q)= d(q,q')= d(q, \Axis(\rho_t(m'_i)))$.

\begin{Lemma}
\label{lemma:glueS} 
Let $r\in \overline{\mathcal S}(q,t)$ belong to the same connected component of the singular locus as $q$. Let $r'$ and $q'$ be the corresponding
closest points in the other components.
If $d(q,r)\leq 10 {{\delta}}_t(q)$, then 
the angle between $qq'$ and $rr'$ after parallel transport (along any of both singular components) is $\leq \gamma(t)$, for some uniform 
$\gamma(t)\to 0$.
\end{Lemma}

This lemma follows easily from the estimates of Lemma~\ref{Lemma:uniformedges} and elementary trigonometric arguments.

\begin{Proposition}
\label{prop:deformreps}
Let $\rho_t$ be as in  Proposition~\ref{assumption:rho}.  There exists $\varepsilon>0$ such that
for $t\in (0,\varepsilon)$ 
there exists $D_t\!:\!\tilde M\to \mathbf H^3$ the developing map of a cone
structure on $(\vert\mathcal O^3\vert,\Sigma_{\mathcal O^3})$ with holonomy
$\rho_t$. In addition, when $t\to 0$, $D_t $ converges to $D_0$,
the developing map of the transverse hyperbolic foliation.
\end{Proposition}

\begin{proof}
Let $0<t<t_0$, where  $t_0>0$ is as in Assertion 3 of Lemma~\ref{Lemma:uniformedges}. The edge $\hat e_1(t)$ is covered by  balls 
$B(q,2 {{\delta}}_{t}(q))$.
Choose a finite covering of such balls, with centers  $q$ in $\hat e_1(t)$.
We claim that  the model $\overline {\mathcal S}(q,t)$ of each ball 
 matches with the next one: this is a consequence of Lemma~\ref{lemma:glueS},
because the segments between $q$ and the opposite singular edge
vary continuously with $q$, and they are almost parallel (the difference
with the parallel transport is uniformly small in $B(q, 10 {{\delta}}_t(q))$). 
 Notice also that  the position
of the singular edges is determined 
by the isometries $\rho_t(m_1)$ and $\rho_t(m'_1)$.
 This gives a metric structure for  a neighborhood of the edges.

By Lemma~\ref{lemma:euctrans}, 
when $q\in \partial \hat e_1(t)$, then 
the $\overline {\mathcal S}(q,t)$ match with the corresponding $\overline {\mathcal R}_i(t)$. In this way we put a geometric structure on 
a solid torus that contains the singular locus, made of the union of 0-cells (the  $\overline {\mathcal R}_i(t)$ for the singular vertices of the polygon)
and $1$-cells (the union of  $\overline {\mathcal S}(q,t)$ for the edges of the polygon).
Let  $D_t$ be the corresponding developing map of this solid torus that contains the singular locus.

Notice that the orientation is globally preserved, by Proposition~\ref{lemma:killingfield},
and
because it depends on the displacement of $\rho_t(f)$.

Recall that we assume that there is no singular fiber in the interior of the orbifold. Look at
the 2-torus that bounds the previous tubular neighborhood of the singularity.
 Now the developing map of the universal covering of the 2-torus factors to a map
from the 2-torus to the hyperbolic solid torus
$\mathbf H^3/\rho_t(f)$, ($\rho_t(f)$ is hyperbolic by
Corollary~\ref{coro:rhotfhyp}).
By Proposition~\ref{lemma:killingfield}, this map is injective on the intersection of the $2$-torus and each model
$\overline{\mathcal S}(q,t)$ and $\overline{\mathcal R}_i(t)$.
In addition, the
models are either far apart or their intersection is  well understood, by
the previous discussion, hence it is an embedding of the torus.

Since it is not contained in a ball, this 2-torus must bound a solid torus in 
$\mathbf H^3/\rho_t(f)$,  with meridian the curve that has trivial holonomy.
This $2$ torus is fibered over a  curve that converges to
the singular locus.
Thus we extend  $D_t$ to the universal covering of the corresponding  solid
torus $V$ in
the smooth part of $\mathcal O^3$. The map
$D_t$ restricted to each compact subset of $\partial \tilde V$ converges to 
$\partial \mathcal P$, coherently with the fibration. Then we choose 
 $D_t$ so that restricted to compact subsets of $\tilde V$ converges to 
the $D_0$.
\end{proof}


\section{Cone manifolds with geometry $\widetilde{SL_2(\mathbf  R)}$ and $\mathbf H^2\times\mathbf R$}
\label{section:weights}

Before explaining the proof of Theorem~\ref{theorem:2}, we give a result about cone manifolds with those fibered geometries, 
just for the statement of the theorem.

As in the introduction, let $\mathcal O^3$ be an orbifold fibering over a polygonal orbifold $P^2$ with mirror boundary and corners.
We assume that $P^2$ has no cone point in the interior, to simplify.
 We will relax the hyperbolicity condition 
for the orbifold $P^2$ by adding cone singularities at the  $I$-fibers. Choose $n$ $I$-fibers of $\mathcal O$, 
$$
\{f_1,\ldots,f_n\}
$$
that include all singular $I$-fibers. Let $q_1,\ldots,q_n\in\mathbf N$ denote their respective indices
in the fibration. In particular
$q_i=1$ if and only if  $f_i$ is a regular fiber.
Fix angles $\vartheta_1,\ldots,\vartheta_n\in (0,2\pi]$ so that 
$$
\vartheta_i/ q_i\leq \pi,
$$
for $i=1,\ldots,n$.
We impose also the following condition
$$
\sum_{i=1}^n(\pi-\vartheta_i/ (2q_i))> 2\pi
$$
this implies that the polygon $Q$ with angles $\vartheta_i/ (2q_i)$ is hyperbolic.

\begin{Proposition}
\label{prop:conemanifolds} Given a hyperbolic structure on $Q$,
there exists a cone manifold $C(\pi)$ with geometry $\widetilde{SL_2(\mathbf  R)}$ or $\mathbf H^2\times\mathbf R$, with the same underlying space as $\mathcal O^3$,
$\Sigma_{C(\pi)}^{Hor}=\Sigma_{\mathcal O^3}^{Hor}$, $\Sigma_{C(\pi)}^{Vert}\subseteq f_1\cup\cdots\cup f_n$, and respective vertical cone angles $\vartheta_1,\ldots,\vartheta_n$, and fibered over $Q$.

In addition, every cone manifold  with geometry $\widetilde{SL_2(\mathbf  R)}$ or $\mathbf H^2\times\mathbf R$ with vertical
 angles $\leq 2\pi$ and with 
space of fibers a polygon with angles $\leq \pi/2$ is obtained in this way.
\end{Proposition}

Since both geometries
$\widetilde{SL_2(\mathbf  R)}$ and $\mathbf H^2\times\mathbf R$ are fibered, for a cone manifold 
with this geometry there is a vertical and a horizontal singular locus, and the horizontal cone angle is always $\pi$.

In the statement, a fiber $f_i$ is in the singular locus of the cone manifold if and only if $\vartheta_i<2\pi$.

\begin{proof}
If the angles are  $\vartheta_i=2\pi/n_i$, then $C(\pi)$ is an orbifold and this is consequence of the geometrization of
Seifert fibered orbifolds (see \cite[Prop. 2.13]{BMP}).

For the general case, we decrease the $\vartheta_i$ to some $\pi/n_i$ and we apply a deformation argument. When the Euler number of the fibration is zero, 
then the geometry of the orbifold involved is $\mathbf H^2\times\mathbf R$, and the geometric structure on $C(\pi)$ is deformed  by deforming  
the basis. Otherwise, the geometric structure is 
 $\widetilde{SL_2(\mathbf  R)}$. 
Let $\tilde{\mathcal O}^3$ be the orientation covering of $\mathcal O^3$, so that there is an orientation reversing involution
$\tau: \tilde{\mathcal O}^3\to\mathcal O^3$ such that $\tilde{\mathcal O}^3/\tau=\mathcal O^3$
and $\textrm{Fix}(\tau)=\Sigma_{\mathcal{O}^3}^{Hor}$.
 The orbifold $\tilde{\mathcal O}^3$ is Seifert fibered, all leaves are circles and the space of leaves is 
 $\tilde Q$,  the sphere with $n$ cone points that is union of $Q$ and the mirror image of $Q$ along the boundary.
The hyperbolic structure on the polygon $Q$ induces a hyperbolic structure on the cone manifold $\tilde Q$.

Let $N=\mathcal O^3\setminus \mathcal{N}(\tilde f_1\cup\cdots\cup\tilde f_n)\cong F\times S^1$, where $F$ is a planar surface with $n$ boundary components.
We first describe the holonomy representation of $N$. Let $a_1,\ldots,a_n\in\pi_1(F)$ denote the peripheral elements so that
$$
\pi_1(F)=\langle a_1,\ldots,a_n\mid a_1\cdots a_n=1\rangle.
$$
Let $f$ be the generator of $\pi_1(S^1)$. The meridian of $\tilde f_i$ is the curve $f^{p_i} a_i ^{q_i}$, where $q_i\geq 1$ is the index of the singular fiber 
(regular when $q_i=1$).

Recall that the identity component of the isometry group of  $\widetilde{SL_2(\mathbf  R)}$ is 
$$
\widetilde{SL_2(\mathbf  R)}\times_{\mathbf Z} \mathbf R,
$$
where $\widetilde{SL_2(\mathbf  R)}$ acts on itself by left multiplication and $\mathbf R$ is the universal covering of $SO(2)$, the stabilizer of 
a point acting on itself by right multiplication of its inverse. For a representation $\rho$ in $\widetilde{SL_2(\mathbf  R)}\times_{\mathbf Z} \mathbf R$,
we denote by $\rho_l$ the projection to $PSL_2(\mathbf  R)$ and $\rho_r$ the projection to $SO(2)=\mathbf R/\mathbf Z$. Notice that $\rho_l(f)=Id$,
because  $PSL_2(\mathbf  R)$ has no center. On the other hand, for $\rho( f^{p_i} a_i ^{q_i})$ to be a rotation
of angle $\vartheta_i$, working in $\mathbf R/2\pi \mathbf Z$, we must have
\begin{equation}
\label{eqn:pqi}
p_i \rho_r(f)+q_i\rho_r(a_i)=\vartheta_i  \textrm{ in }   \mathbf R/2\pi \mathbf Z,   \qquad \textrm{ for } i=1,\ldots,n.
\end{equation}
Combining this with
\begin{equation}
\label{eqn:as0} 
\rho_r(a_1)+\cdots+ \rho_r(a_n)=0  \textrm{ in }   \mathbf R/2\pi \mathbf Z,
\end{equation}
it follows that $\rho_r(f)$ and $\rho_r(a_i)$ are locally uniquely determined in $\mathbf R$, because the Euler number
does not vanish:
$$
\frac{p_1}{q_1}+\cdots +\frac{p_n}{q_n}\neq 0.
$$
The reason is that this Euler number is the determinant of the matrix associated to the linear system (\ref{eqn:pqi}) and (\ref{eqn:as0}).

Now we describe the deformation argument. By changing the angles and the hyperbolic structure of $Q$, Equations (\ref{eqn:pqi}) and (\ref{eqn:as0})
imply that we can deform the representation of $\pi_1(N)$ in $\widetilde{SL_2(\mathbf  R)}\times_{\mathbf Z} \mathbf R$, in such a way that
the meridians go to the rotation of the expected angle, this gives a $\widetilde{SL_2(\mathbf  R)}$ cone structure on $\tilde {\mathcal O}^3$ 
with the deformed cone angles on the 
$\tilde f_i$.
 Moreover, since the solution to  (\ref{eqn:pqi}) and (\ref{eqn:as0}) is localy unique and the metric structure in $\tilde Q$ is invariant by the involution,
$\tau$ is homotopic to the isometry. By applying Tollefson's theorem to $N\cong F\times S^1$, $\tau$ is conjugate to an isometry,
giving the singular
structure on $\mathcal O^3$. This proves openness for the deformation. For closedness, we need to show that $\rho(f)$ will not become trivial. By contradiction, assume that $\rho(f)=0$,
then by (\ref{eqn:pqi}) $\rho_r(a_i)=\vartheta_i/q_i$ $\mod 2\pi\mathbf Z$  and by (\ref{eqn:as0}) $\sum \vartheta_i/q_i\in 2\pi\mathbf Z$. Then we have to look carefully at the determinations in the universal 
covering $\widetilde{SL_2(\mathbf  R)}\times \mathbf R$. Choose a lift of $\rho$ such that $\tilde \rho_r(a_i)=\vartheta_i/q_i\in (0,2\pi)$. Then, by a deformation argument and viewing
the Euclidean case as a limit case, we get that $\tilde \rho_l(a_1\cdots a_n)$ is a lift of a rotation of angle  $2\pi(n-2)$. On the other hand
$\tilde \rho_r(a_1\cdots a_n)$ lifts to $\sum \vartheta_i/q_i\in\mathbf R$, but $\sum \vartheta_i/q_i<2\pi(n-2)$, hence $\tilde\rho(a_1\cdots a_k)$ cannot be the lift or a trivial element in 
in 
 $\widetilde{SL_2(\mathbf  R)}\times_{\mathbf Z} \mathbf R$, which leads to contradiction.
\end{proof}

Let us explain now how to adapt the proof of Theorem~\ref{theorem:1} to Theorem~\ref{theorem:2}.

We use the analogue results of Section~\ref{section:adding} to construct a curve of representations
when all $\vartheta_i<2\pi$. When some $\vartheta_i=2\pi$, then we require the analogue deformation argument of 
Section~\ref{section:curve}. 

The arguments of Section~\ref{section:adding} work exactly the same, just by replacing the vector $(1,\ldots,1)$ by 
$(w_1,\ldots,w_n,w_1',\ldots,w_n')$.

In Section~\ref{section:curve} one has to work with real analytic sets instead of algebraic ones, but all results apply.
Namely, 
Sullivan's local Euler characteristic theorem in the proof of Corollary~\ref{coro:chi0inS} is already stated for real analytic sets.

The required result on the Teichm\"uller space is  Corollary~\ref{cor:geometry}.

Regarding the construction of developing maps, Section~\ref{section:Killing} and Section~\ref{section:developingmaps} 
apply with no changes.

\section{An example}
\label{section:example}

Let $\mathcal O^3$ denote the orbifold with underlying space $S^3$ and singular locus the Whitehead link.
Assume that the respective labels in the singular components are $n>4$ and $2$. This orbifold is Seifert fibered: 
 the component with label $n>4$ is a fiber,
and the one with label $2$ is the union of mirror points of the $I$-fibers, and projects to mirror points of $P^2$.
The base of the Seifert fibration is a one-edged polygonal orbifold, with a single corner, and its interior  
contains a cone point with label $n$, corresponding to the singular component that is also a fiber. 
The angle at the corner is $\pi/2$, cf.\ Figure~\ref{fig:witehead}.

\begin{figure}
\begin{center}
{
\psfrag{n}{$n$}
\psfrag{2}{$2$}
\psfrag{p}{$P^2$}
\psfrag{o}{$\mathcal O^3$}
\psfrag{pm}{$\pi/2$}
\includegraphics[height=4cm]{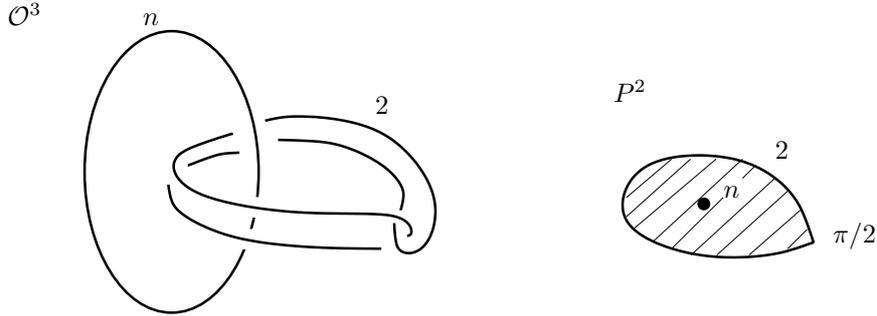}
}
\end{center}
   \caption{The orbifold $\mathcal O^3$ with singular locus the Whitehead link and the  base of the fibration $\mathcal P^2$.}\label{fig:witehead}
\end{figure}

In our proof of Theorem~\ref{theorem:1} we   work with the $n$-th branched covering
(thus $P^2$ lifts to the regular $n$-sided right-angled polygon), but 
for explicit computations it is easier  to work in $\mathcal O^3$ instead of its covering.

The smooth part of $\mathcal O^3$ is $M=\mathcal O^3\setminus \Sigma_{\mathcal O^3}$, 
the complement of the Whitehead link in $S^3$.  If $a, b \in \pi_1M$
are meridians around the two components of $\Sigma$, the
fundamental group has the following presentation, cf.~\cite{HLM}:
$$
\pi_1(M) = \langle a,b \,\vert\, awa^{-1}w^{-1} \rangle
$$
with $w=bab^{-1}a^{-1}b^{-1}ab$.
The $SL_2(\mathbf C)$-character variety of $M$ has been computed by \cite{HLM}.
Here we follow the exposition from \cite{PortiWeiss}.
Namely, after identifying $X(M)$ with the image of the map
$$
\begin{array}{rl}
 X(M) &\rightarrow \mathbf C^3\\
\chi & \mapsto (\chi(a),\chi(b),\chi(ab))
\end{array}
$$
in $\mathbf C^3$, it is 
$$
X(M)=\{ (x,y,z) \in \mathbf C^3 \mid p(x,y,z) \cdot q(x,y,z)=0\}
$$
with
$$
p(x,y,z) = xy - (x^2+y^2-2)z+xyz^2 - z^3
$$
and
$$
q(x,y,z)=x^2+y^2+z^2-xyz-4\,.
$$

Again by \cite{HLM}, $q=0$ corresponds to the abelian characters and $p=0$ is the closure
of the set of irreducible characters.

We work with the subvariety $y=\pm 2\cos(\pi/n)$, with $n>4$. The holonomies of the deformations of
Proposition~\ref{assumption:rho}
are obtained by 
taking $x=\pm 2\cos(\alpha/2)$, for $\alpha=\pi-t$ close to $\pi$.

The point of the initial holonomy $\chi_0$ has coordinates 
$$x=0, \ y= \pm 2\cos(\pi/n)\textrm{ and }z=\pm \sqrt{4\cos^2(\pi/n)-2}\in \mathbf i\mathbf R.
$$
The sign ambiguity comes from different lifts to $SL_2(\mathbf C)$.
Since $n>4$,
$$
\frac{\partial p}{\partial z}(\chi_0)\neq 0.
$$
 Hence, by the implicit function theorem, $x$ is  a local
parameter of the variety
$$
p(x, \pm 2\cos(\pi/n),z)=0
$$
around $\chi_0$.
In particular $\alpha=\pi-t$ is a local parameter of our deformation space.
Recall that changing the sign of $t$ corresponds to changing the orientation.

This example was the conjectural picture for a piece of the boundary of the moduli space of hyperbolic cone structures of \cite{PortiWeiss}.

Other regular or singular Dehn fillings on the same component of the Whitehead link give a Seifert fibered orbifold with the same base.
Here we just described the $1/0$-Dehn filling  with singular core with ramification $n\geq  5$, but we can also consider the $p/q$-Dehn filling
with $p \geq 1$ either with regular core ($n=1$) or with singular core of ramification $n\geq 2$. The base is the same orbifold, 
and the cone point has label $p\, n$. So the base is hyperbolic when $p\, n> 4$. Similar explicit computations of the variety of characters can be  made, but they are more involved.

\appendix
\section{Appendix: Earthquakes and polygons}
\label{sec:appendix}

We  are in the setting of Lemma~\ref{lemma:compact}. Thus, assuming that $P^2$ is
a polygon with cone angles $\leq \vartheta_0/2<\pi/2$, we take its double
$C$, which is a cone manifold with underlying space a sphere, and singular locus
$n$ points of cone angle $\leq \vartheta_0<\pi$.

We follow  the work of Bonsante and Schlenker in 
\cite{BonsanteSchlenker} about earthquakes for closed cone manifolds with cone angles $<\pi$.

The Teichm\"uller space of $C$ is denoted by $\mathcal T(C)$, and we assume that
cone angles are fixed.
Let $\mathcal{ML}(C)$ and $\mathcal{PML}(C)$ denote the spaces of measured and
projectively measured geodesic laminations on the smooth part of $C$.
Thus $\mathcal{ML}(C)\cong \mathbf R^{2n-6}$ and $\mathcal{PML}(C)\cong S^{2
n-7}$. Lemma~2.2 of \cite{BonsanteSchlenker} proves that any lamination of
$C$
can be realized by a geodesic lamination. The main result of 
\cite{BonsanteSchlenker} for our purposes is
the earthquake theorem for cone manifolds with cone angles $<\pi$:

\begin{Theorem}[\cite{BonsanteSchlenker}]
\label{thm:earthquake}
For any $S\in \mathcal T(C)$, the map
\begin{eqnarray*}
 \mathcal{ML}(C) & \to & \mathcal T(C) \\
  \lambda  & \mapsto & E_{\lambda}^l(S)
\end{eqnarray*}
is a homeomorphism,
where $E_{\lambda}^l$ denotes the left earthquake of $S$ along
$\lambda\in\mathcal{ML}(C)  $.
\end{Theorem}

We also require a version of Wolpert's formulas \cite{Wolpert} in our setting,
namely for the lengths
of the segments between cone points. 
Let $\sigma$ denote a geodesic segment of $C$ between two cone points. The
perimeter we are interested in is the addition of lengths  of such geodesic segments.
The length of $\sigma$ is denoted by $\vert\sigma\vert$.

\begin{Lemma} 
\label{lemma:wolpertlamination}
Let $S\in \mathcal T(C)$, $\lambda\in \mathcal{ML}(C)$  and $\sigma$ a geodesic
segment between two cone points.
For   the earthquake deformation 
$E_{ t \lambda}^l(S)$, $t\in\mathbf R_+$:
\begin{itemize}
 \item [(1)] $\frac{\partial \vert \sigma\vert}{\partial t} =  \int_{\lambda}
\cos \theta$.
\item [(2)] $\frac{\partial^2 \vert \sigma\vert}{\partial t^2} \geq c \langle
\lambda, \sigma\rangle^2   $, for some uniform $c>0$.
\end{itemize}
\end{Lemma}

Here  $\theta$ denotes the angle from $\lambda$ to $\sigma$ measured
counterclockwise at each intersection point, 
$\int_{\lambda} \cos \theta$ means the integral of the cosine of this  angle on
the intersection between $\sigma$ and $\lambda$ 
along the transverse measure of $\lambda$, and 
$$
\langle \lambda, \sigma\rangle=\int_{\lambda} \chi_{\sigma}
$$ 
is
the intersection between $\lambda$ and $\sigma$ (which is the integral of the
characteristic function of $\sigma$ along the
transverse measure of $\lambda$).

To prove this lemma we use the density of weighted multicurves in 
$\mathcal{ML}(C)$ and reduce to the following computation:

\begin{Lemma}[Wolpert's formulas] 
\label{lemma:wolpertscc}
Let $S\in \mathcal T(C)$, $\lambda\in \mathcal{ML}(C)$  and $\sigma$ a geodesic
segment between two cone points.
For  the earthquake deformation 
$E_{ t \lambda}^l(S)$, $t\in\mathbf R_+$:
 \begin{itemize}
 \item [(1)] Let $\lambda$ be a simple closed curve. Then $\frac{\partial \vert
\sigma\vert}{\partial t} = \int_{\lambda} \cos \theta$.
\item [(2)] 
Let $\lambda=\sum m_i\lambda_i$, be a weighted multicurve, with $\lambda_i$
disjoint simple closed curves and $m_i\in\mathbf R_ +$.
 Then
$$
\frac{\partial^2 \vert \sigma\vert}{\partial t^2} =
\sum_{ij} m_i m_j \sum_{(p,q)\in(\sigma\cap\lambda_i)\times
(\sigma\cap\lambda_j)}
\frac{\cosh\vert \sigma_{pq}^+\vert \cosh\vert \sigma_{pq}^-\vert}{\sinh{\vert 
\sigma\vert}} \sin\theta_p\sin\theta_q
$$
where $\sigma_{pq}^+$ and $\sigma_{pq}^-$ are the two components of
$\sigma\setminus\overline{pq}$, and 
 $\theta_p$ and $\theta_q$ are the angles at the corresponding intersection
points. 
\end{itemize}
\end{Lemma}

 Lemma~\ref{lemma:wolpertscc}(1) 
implies
Lemma~\ref{lemma:wolpertlamination}(1)
by density of multicurves in $\mathcal{ML}(C)$. For item (2), one uses in
addition the fact that the cone 
angles are $\leq\vartheta_0<\pi$. This implies that the distance between cone
points and $\lambda_j$ is 
$\geq C(\vartheta_0)>0$, for some uniform $C(\vartheta_0)$ depending on $\vartheta_0<\pi$ 
(see Remark~2 in \cite{BonsanteSchlenker}). Thus 
$\sin\theta_p$ and $\sin\theta_q$ are also uniformly bounded away from zero.

\begin{proof}[Proof of Lemma~\ref{lemma:wolpertscc}]
 To prove (1), consider the result of an earthquake of length $t$ as in
Figure~\ref{fig:earthquake1} and analyze limits when $t\searrow 0$.

\begin{figure}
\begin{center}
{\psfrag{l}{$\lambda$}
\psfrag{s}{$\sigma$}
\psfrag{p}{$p$}
\psfrag{p0}{$p_0$}
\psfrag{p1}{$p_1$}
\psfrag{p2}{$p_2$}
\psfrag{c1}{$c_1$}
\psfrag{c2}{$c_2$}
\psfrag{l1}{$l_1$}
\psfrag{l2}{$l_2$}
\psfrag{t}{$\theta$}
\includegraphics[height=4cm]{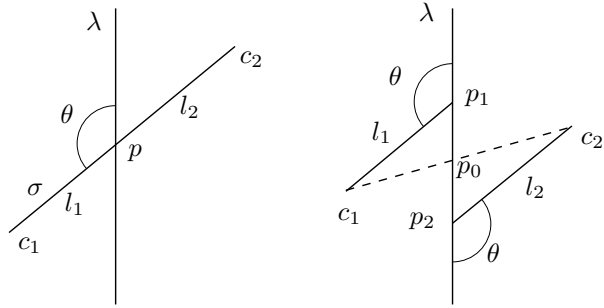}
}
\end{center}
   \caption{The earthquake of parameter $t=\vert p_1p_2\vert>0$ transforms the
picture on the left to the picture on the right.}\label{fig:earthquake1}
\end{figure}

In the picture, $\sigma$ is the segment between two cone points $c_1$ and $c_2$,
and it is 
divided into two segments of length $l_1=\vert c_1p\vert$ and $l_2=\vert
c_2p\vert$ respectively. 
Let $\alpha_1=\angle p_1c_1p_0$ and $\alpha_2=\angle p_2c_1p_0$ denote the
angles between the old and the new 
minimizing segments between $c_1$ and $c_2$ after the earthquake, at $c_1$ and $c_2$
respectively.
Of course $\alpha_1\to 0$ and $\alpha_2\to 0$ as $t\searrow 0$.
Accordingly, we divide the segment $p_1p_2$ of length  $t$ into two pieces
$p_1p_0$ and $p_0p_2$ of length $t_1$ and $t_2$ respectively. 
So we have $t=t_1+t_2$ and $\vert\sigma\vert=l_1+l_2$. Writing $l_1+\delta
l_1=\vert c_1p_0\vert$, 
the cosine formula gives
$$
\cosh(l_1+\delta l_1)= \cosh l_1\cosh t_1 - \sinh l_1\sinh t_1\cos(\pi-\theta).
$$
Moreover, by the triangle inequality $\vert \delta l_1\vert \leq t_1$, hence
\begin{equation}
\label{eqn:l1t1}
\lim_{t_1\searrow 0} \frac{\delta l_1}{t_1}= \lim_{t_1\searrow 0} \frac{\cosh(l_1+\delta
l_1)-\cosh l_1}{t_1\sinh l_1}=-\cos(\pi-\theta)=\cos\theta. 
\end{equation}

Define $\theta+\delta\theta$ to be the angle between $\lambda$ and the segment 
from $c_1$ to $c_2$
after the earthquake.
Using the hyperbolic sine formula:
$$
\frac{\sin( \theta+\delta\theta)}{\sinh(l_1)}=\frac{\sin( \pi-\theta)}{\sinh(l_1+\delta l_1)},
$$
we have
\begin{multline*}
 \lim_{t_1\searrow 0}\frac{\sin(\theta+\delta\theta) - \sin (\theta) }{t_1}
=
\lim_{t_1\searrow0}\sin(\theta)\frac1{t_1}\frac{\sinh(l_1)-\sinh(l_1+\delta l_1)  }{\sinh(l_1+\delta l_1)}
=\\
-\frac{\sin(\theta)\cos(\theta)}{\tanh(l_1)}
\end{multline*}
Thus 
\begin{equation}
\label{eqn:t1t2}
\lim_{t\searrow 0}\frac{t_1}{t_2}=\frac{\tanh l_1}{\tanh l_2},
\end{equation}
which proves that $t_1$ and $t_2$ are  infinitesimals of $t$ of the same order.
Notice that we have given an argument for (\ref{eqn:t1t2}) when $\cos\theta\neq 0$, but when $\theta=\pi/2$, the picture is symmetric and then $t_1=t_2=t/2$
and $l_1=l_2$.

Equation~(\ref{eqn:l1t1}) and the fact that  $t_1$ and $t_2$ are  infinitesimals of $t$ of the same order prove the first formula of
the lemma when the curves meet at a single point, 
and the general case follows from linearity.

Notice that from Equation (\ref{eqn:t1t2}) we also have:
\begin{equation}
\label{eqn:t2t}
\frac{t_2}{t}\to \frac{\tanh l_2}{\tanh l_1 + \tanh l_2}=\frac{\sinh l_2\cosh
l_1}{\sinh\vert\sigma\vert}.
\end{equation}
Now we consider a new geodesic $\mu$ in the previous picture, and we want to
estimate the derivative of $\cos(\eta)$, where $\eta$ is the angle at $q=\mu\cap\sigma$ from
$\mu$ to $\sigma$ (counterclockwise).
In Figure~\ref{figure:earthquake2},  $\eta+\delta\eta$ is the angle at $q_0$
(the intersection of $\mu$ with the segment $c_1c_2$ after the earthquake).

\begin{figure}
\begin{center}
{\psfrag{l}{$\lambda$}
\psfrag{mu}{$\mu$}
\psfrag{s}{$\sigma$}
\psfrag{p}{$p$}
\psfrag{q}{$q$}
\psfrag{q0}{$q_0$}
\psfrag{q}{$q$}
\psfrag{p0}{$p_0$}
\psfrag{p1}{$p_1$}
\psfrag{p2}{$p_2$}
\psfrag{c1}{$c_1$}
\psfrag{c2}{$c_2$}
\psfrag{l1}{$l_1$}
\psfrag{l2}{$l_2$}
\psfrag{l3}{$l_3$}
\psfrag{t}{$\theta$}
\psfrag{e}{$\eta$}
\includegraphics[height=4cm]{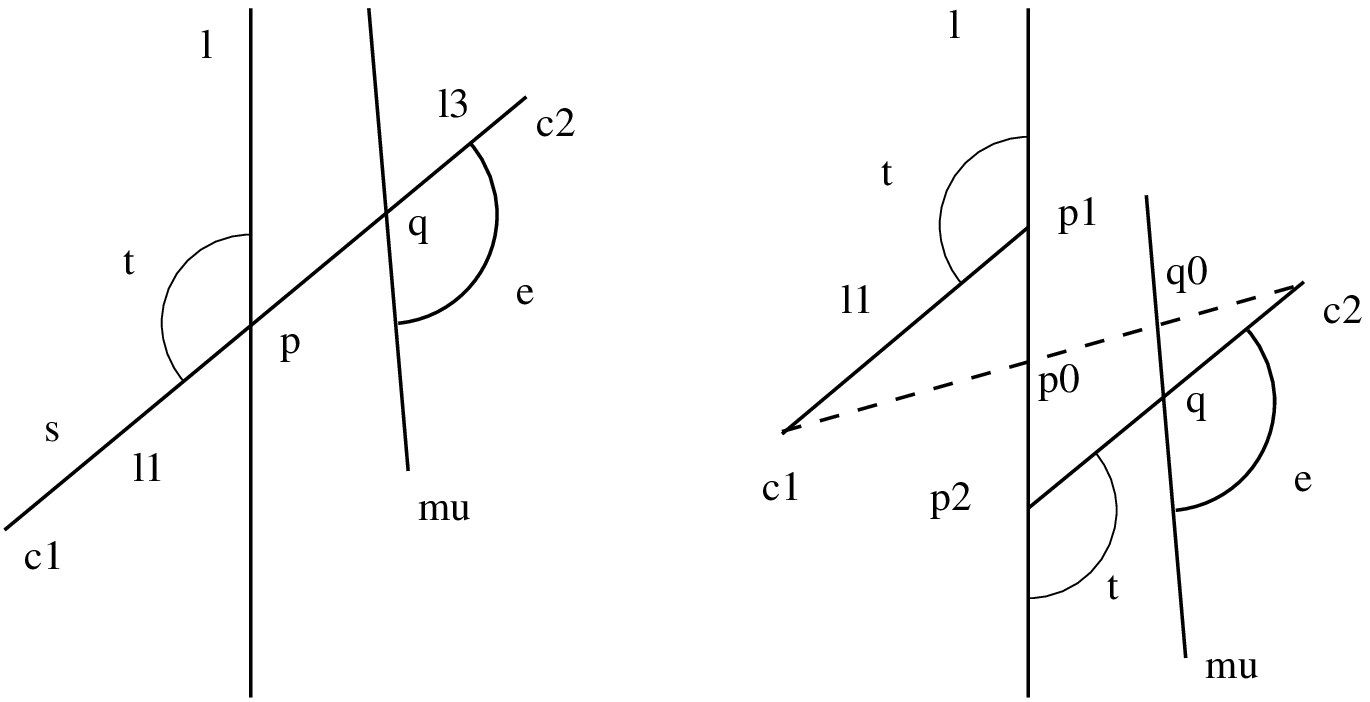}
}
\end{center}
   \caption{The earthquake of parameter $t=\vert p_1p_2\vert>0$ transforms the
picture on the left to the picture on the right.} \label{figure:earthquake2}
\end{figure}

First, applying the sine formula, we have:
\begin{equation}
\label{eqn:alpha2t2} 
\lim_{t_2\searrow 0}\frac{\alpha_2}{t_2}= \lim_{t_2\searrow 0}\frac{\sin \alpha_2}{\sinh
t_2}=\lim_{t_2\to 0} \frac{\sin(\pi-\theta)}{\sinh(l_2+\delta l_2)}=
\frac{\sin\theta}{\sinh l_2}.
\end{equation}
By applying the dual cosine formula:
\begin{eqnarray*}
 \cos(\eta+\delta\eta) &= &
-\cos(\pi-\eta)\cos\alpha_2+\sin(\pi-\eta)\sin\alpha_2\cosh l_3\\
	&=& \cos\eta\cos\alpha_2+\sin\eta\sin\alpha_2\cosh l_3,
\end{eqnarray*}
where $l_3=\vert qc_2\vert$. Thus
\begin{multline}
\label{eqn:limcos}
\lim_{t\searrow 0}\frac{\cos(\eta+\delta\eta)-\cos\eta}{t}=\lim_{t\searrow
0}\frac{\cos(\eta+\delta\eta)-\cos\eta\cos\alpha_2}{t}=\\
\lim_{t\searrow 0}\frac{\sin\eta\sin\alpha_2\cosh l_3}{t}.
\end{multline}
Combining   Equations (\ref{eqn:t2t}), (\ref{eqn:alpha2t2}) and (\ref{eqn:limcos}), we get:
$$
\lim_{t\searrow 0}\frac{\cos(\eta+\delta\eta)-\cos\eta}{t}= \frac{\sin\theta\sin\eta
\cosh l_3 \cosh l_1}{\sinh\vert\sigma\vert}.
$$
Writing $\theta=\theta_p$, $\eta=\theta_q$, $l_1=\sigma_{pq}^+$ and
$l_3=\sigma_{pq}^-$, we get that, 
if we only deform the earthquake in a neighborhood of $p$, then
\begin{equation}
 \label{eqn:derivativecosine}
\frac{d \phantom{t}}{dt} \cos(\theta_q) =\frac{\cosh\vert \sigma_{pq}^+\vert
\cosh\vert \sigma_{pq}^-\vert}{\sinh{\vert  \sigma\vert}} \sin\theta_p\sin\theta_q.
\end{equation}
We next deduce the formula for a multicurve $\lambda=\sum m_i \lambda_i$ from
(\ref{eqn:derivativecosine}), 
where each $\lambda_i$ is a simple closed curve and $m_i\in\mathbf R_+$. 
The first derivative is:
$$
\frac{\partial\vert\sigma\vert}{\partial E^l_{t\lambda}}=\sum_i m_i
\frac{\partial\vert\sigma\vert}{\partial E^l_{t\lambda_i}}=
\sum_i m_i \sum_{p\in {\lambda_i\cap \sigma}} \cos\theta_p.
$$
In this formula $\partial/\partial E^l_{t\lambda}$ and $\partial/\partial
E^l_{t\lambda_i}$ are used to distinguish the foliations
of the earthquakes.
Hence 
$$
\frac{\partial^2\vert\sigma\vert}{(\partial E^l_{t\lambda})^2}= 
\sum_{ij} m_i m_j  \sum_{p\in {\lambda_j\cap \sigma}}  
\frac{\partial\cos\theta_p }{\partial E^l_{t \lambda_i}} .
$$

For each $p\in {\lambda_j\cap \sigma} \lambda_j$, we compute the derivative of
$\cos\theta_p$ with respect to 
$\partial E^l_{t\lambda_i}$ by regarding the contribution
of all terms in the intersection $\lambda_i\cap \sigma$ and by applying
(\ref{eqn:derivativecosine}), proving the lemma.
\end{proof}

We also need the following lemma. It is due to Kerckhoff for smooth surfaces, but his
proof in \cite[Thm.~3.5]{KerckhoffDuke} applies verbatim here:

\begin{Lemma}
 \label{lemma:kerckhofftg}
Every tangent vector to $\mathcal T(C)$   is tangent to an earthquake
map.
\end{Lemma}

Combining Lemmas~\ref{lemma:wolpertscc} and~\ref{lemma:kerckhofftg} with
Theorem~\ref{thm:earthquake}, we get that the Hessian of the perimeter in $\mathcal T(C)$
is positive definite. We deduce then the following corollary, which implies Lemma~\ref{lemma:compact}.

\begin{Corollary}
\label{cor:minimizer}
 The perimeter of $P^2$ has a unique minimum in $\mathcal T(C)$. Moreover,
the determinant of the Hessian of the perimeter is nonzero on this minimum,
in particular it is an isolated critical point.

The same holds true on  $\mathcal T(P^2)$.
 \end{Corollary}

To prove the  assertion for  $\mathcal T(P^2)$, notice that not all cone manifold structures on $ C$ are doubles of polygons.
In fact $\mathcal T(P^2)\subset \mathcal T(C)$ is a subspace invariant by the involution on  $\mathcal T(C)$
induced by the involution on $\mathcal C$.
 The minimizer of the perimeter in  $\mathcal T(C)$ must be symmetric by uniquenes, hence it lies in $\mathcal T(P^2)$.
Moreover, as the minimizer is a critical point, the Hessian restricted to $\mathcal T(P^2)$ must be also positive definite.

Here we restate and prove Proposition~\ref{prop:geometry}.

\begin{Corollary}
\label{cor:geometry}
Given $0<\vartheta_1,\ldots,\vartheta_n\leq  \pi/2$ with $\sum (\pi-\vartheta_i)>2\pi$, there exists a unique polygon 
$\mathcal P$ in $\mathbf H^2$ with those angles that minimizes its perimeter.
This structure is an isolated critical point for the perimeter in the Teichm\" uller space of the polygon.

In addition, $\mathcal P$ is the only polygon with those angles that has an inscribed circle tangent to all 
of its edges.

The existence and uniqueness of a minimizer of the $\mathcal W$-perimeter
 holds true for any choice of weights $\mathcal W=\{w_1,\ldots,w_n\}$, $w_i>0$.
In this case, $\mathcal P$ is the only polygon with an interior point $p$ such that $\frac1{w_i}\sinh(d(p,e_i))$
is independent of the edge $e_i$ of $P$.
\end{Corollary}

\begin{proof}
When $\vartheta_1,\ldots,\vartheta_n<\pi/2$, existence and uniqueness is  
Corollary~\ref{cor:minimizer}. The fact that it has an inscribed circle tangent to all 
of its faces (when there are no weights)  
 follows from the properties of the Killing vector field in Proposition~\ref{lemma:killingfield}.
When there are weights $\mathcal W$, this assertion is consequence of the formulas in the proof of Proposition~\ref{lemma:killingfield}.

We discuss the case when some of the angles $\vartheta_i$ become $ \pi/2$. 
Assume first that at least two of the cone angles become $\pi/2$. Then the cone manifold $C$ obtained as a double of the polygon has at least
two cone angles equal to $\pi$. Consider a two to one ramified covering of $C$, with ramification locus precisely the cone points of angle $\pi$.
In this way we obtain a cone manifold with cone angles $<\pi$, and we may apply the previous argument, also  with arcs and some closed geodesics, instead of only arcs.

Assume now there is only one cone angle equal to $\pi/2$ and all other cone angles are $\leq \vartheta_0/2<\pi/2$.
Then the cone manifold obtained as double of the polygon has one cone angle $\pi$ and all other cone angles $\leq\vartheta_0<\pi$.
The Teichm\"uller space of the polygon is embedded in $\mathbf R^n_+$, with coordinates the length of the edges.
Using the ambient metric of $\mathbf R^n_+$, a unit tangent vector $v$ to this Teichm\"uller space
has at least one coordinate larger than  $1/\sqrt{n}$ (up to replacing $v$ by $-v$).
We can assume that it happens to be the first coordinate. 
Now, deform the angle $\pi/2$ to $\pi/2-\varepsilon$, and call $\tau_{\varepsilon}$
the minimum, that varies continuously. 
Deform also the tangent vector $v_{\varepsilon}$ 
so that the first coordinate is at least $1/(2\sqrt{n})$. 
View $v_{\varepsilon}$ as a tangent vector $\tilde v_{\varepsilon}$ to the Teichm\"uller space of $C$, 
the cone manifold double of the polygon.
Let $E^l_{ t \lambda}$ be the earthquake path tangent to $\tilde v_{\varepsilon}$. 
By Lemma~\ref{lemma:wolpertlamination} (1), the intersection of $\lambda$ with the first edge is $\langle\lambda,\sigma_1\rangle\geq 1/(2\sqrt{n})$. Thus $\langle\lambda,\sigma\rangle\geq 1/(2\sqrt{n})$ and,
by Lemma~\ref{lemma:wolpertlamination} (2), the second derivative of $\vert\sigma\vert$ in the direction of $v$ is $\geq c/(2\sqrt{n})$, for some 
$c>0$ that we claim that exists and is uniform because the cone manifold $C$ has only one cone angle equal to $\pi$. 
Once we have this claim, there is a positive lower bound to the second derivative of the perimeter, that it is uniform on $\varepsilon>0$, hence
it holds for $\varepsilon=0$ and shows that the Hessian of the perimeter is positive definite.

Let us prove the existence of this  $c>0$ uniform in $\varepsilon>0$,
 when precisely one of the cone angles is $\pi-\varepsilon$ and the other ones are 
$\leq\vartheta_0<\pi$.  
The distance between a cone point of angle $\leq \vartheta_0<\pi$ and a closed geodesic is
$\geq C(\vartheta_0)>0$. Thus,
if $\sigma_i$ is a segment between two cone points, since at least one of them has angle $\leq\vartheta_0<\pi$,  
then the angle between any closed geodesic
and $\sigma_i$ is bounded below in terms of $\vartheta_0$ and an upper bound to the  length $\vert\sigma\vert$, 
hence the angles $\theta_p$ and $\theta_q$ that occur in   Lemma~\ref{lemma:wolpertscc}(2) are bounded
below away from zero, independently of $\varepsilon$. This gives the desired bound.
\end{proof}

The arguments in the proof of Corollary~\ref{cor:geometry} also give:

\begin{Remark}
\label{Remark:pi}
Corollary~\ref{cor:minimizer} also holds true when some angles of $P^2$ are $\pi/2$.
\end{Remark}

\section{Appendix: Infinitesimal isometries}
\label{sec:inf}

The Lie algebra of infinitesimal isometries is denoted by  $\mathfrak{sl}_2(\mathbf C)$. It is naturally equipped with the complex Killing form 
$$
B\! : \! \mathfrak{sl_2}(\mathbf C)\times \mathfrak{sl_2}(\mathbf C)\to \mathbf C
$$
defined by $B(\mathfrak a,\mathfrak b)= \operatorname{Trace}(Ad_{\mathfrak a}\circ Ad_{\mathfrak b}) =
4 \operatorname{Trace}( {\mathfrak a} {\mathfrak b})$.
Thus,
$$
B\left(
	\begin{pmatrix}
	 a & b \\
	c & -a
	\end{pmatrix},
	\begin{pmatrix}
	 x & y \\
	z & -x
	\end{pmatrix}
\right)= 8 ax+4 bz+4cy.
$$
Let $\mathfrak{a}$ be an infinitesimal isometry of complex length $l$ (ie.\ $\exp(t\mathfrak{a})$ has complex length
$t \, l$). Then 
$$
l=\pm\frac{1}{\sqrt{2}}\sqrt{B(\mathfrak{a},\mathfrak{a})}.
$$
In particular $B(\mathfrak{a},\mathfrak{a})=0$ iff $\mathfrak{a}$ is trivial or parabolic.

The Killing vector field $F$ corresponding to $\mathfrak a\in\mathfrak{sl_2}(\mathbf C)$ is the field tangent to the orbits
of $\exp(t\mathfrak {a})$, the 
one parameter group of diffeomorphism of $\mathbf H^3$.
We notice that for $x\in\mathbf H ^3$, $F_x$ is the translational part of $\mathfrak a$ at $x$.
When $\mathfrak a$ is not parabolic, then 
$\exp(t\mathfrak {a})$ has an invariant axis, that it is also the minimizing locus for the norm $\vert F\vert$. This axis
is denoted by $\Axis(\mathfrak{a})$. 
The Killing vector field has nonempty vanishing locus iff $\mathfrak a$ is an infinitesimal rotation, then it vanishes precisely at 
$\Axis(\mathfrak{a})$.

When $\mathfrak a$ is parabolic, then $\exp(t\mathfrak {a})$ fixes a point at $\infty$,
that we denote by
 $\Fix(\mathfrak{a})\in\partial \mathbf H^3$.

Following 
Fenchel~\cite{Fenchel},
we denote by $ d_{\mathbf C}$ the complex distance between two geodesics, ie. the real part is the metric distance
and the imaginary part the rotation angle.

\begin{Proposition}
\label{prop:killing}
Let $\mathfrak{a}, \mathfrak{b}\in\mathfrak{sl}_2(\mathbf C)$ be two nonzero and nonparabolic infinitesimal isometries. Then
$$
\frac{B(\mathfrak{a},\mathfrak{b})^2}
{ 
{B(\mathfrak{a},\mathfrak{a})B(\mathfrak{b},\mathfrak{b})}
} 
=
\cosh^2 d_{\mathbf C}(\Axis(\mathfrak{a}),\Axis(\mathfrak{b})) .
$$
 \end{Proposition}

\begin{proof}
Notice that $B(\mathfrak{a},\mathfrak{a})= -8 \det(\mathfrak{a})$. Thus  since traceless matrices in $SL_2(\mathbf C)$ are $\pi$-rotations in $\mathbf H^3$,
 $$
\sqrt{\frac{-8}{B(\mathfrak{a},\mathfrak{a})}}\,\mathfrak{a}\in SL_2(\mathbf C)
$$
is a  rotation of angle $\pi$ around $\Axis(\mathfrak{a})$. Hence, the product 
 $$
\sqrt{\frac{-8}{B(\mathfrak{a},\mathfrak{a})}}\,\mathfrak{a} 
\sqrt{\frac{-8}{B(\mathfrak{b},\mathfrak{b})}}\,\mathfrak{b} 
=
\frac{\pm 8}{\sqrt{B(\mathfrak{a},\mathfrak{a}) B(\mathfrak{b},\mathfrak{b}) }}\, \mathfrak{ab}
$$
is an isometry of complex length $2 d_{\mathbf C}(\Axis(\mathfrak{a}),\Axis(\mathfrak{b}))$. Thus
$$
\operatorname{Trace}( 
\frac{\pm 8}{\sqrt{B(\mathfrak{a},\mathfrak{a}) B(\mathfrak{b},\mathfrak{b}) }} \, \mathfrak{ab}
) 
=
\pm 2\cosh d_{\mathbf C}(\Axis(\mathfrak{a}),\Axis(\mathfrak{b}))
$$
The proposition follows from this formula and 
$
 B(\mathfrak{a},\mathfrak{b})= 4 \operatorname{Trace}(\mathfrak{a}\mathfrak{b}).
$
\end{proof}

The idea of considering elements of the Lie algebra that are also in $SL_2(\mathbf C)$ as rotations of angle $\pi$  
is taken from the so called line geometry in  Marden's book \cite[Ch.\ 7]{Marden}.

\begin{Remark}
\label{remark:killing}
In the previous proposition, if   $\mathfrak a,\mathfrak b\in\mathfrak {sl}_2(\mathbf C)$ are infinitesimal 
rotations of respective angles $\alpha$ and $\beta$, and if we have $\alpha,\beta>0$, then it makes sense to talk about 
orientation of their axis.
In this case we have:
\begin{equation}
\label{eqn:killingnp}
B(\mathfrak{a},\mathfrak{b})= -2\alpha\beta \cosh d_{\mathbf C}(\Axis(\mathfrak{a}),\Axis(\mathfrak{b})). 
\end{equation}
\end{Remark}

This remark follows immediately from Proposition~\ref{prop:killing} and a continuity argument, 
by deforming first $\beta$ to $\alpha$ and then by moving one of the oriented edges to the other,
because 
$B(\mathfrak{a},\mathfrak{a})= -2 \alpha^2$.

\medskip

Given four points $z_1,z_2,z_3,z_4\in\partial\mathbf H^3=\mathbf C\cup\infty$, the \emph{cross ratio} is
$$
[z_1:z_2:z_3:z_4]=\frac{(z_1-z_3)(z_2-z_4)}{(z_2-z_3)(z_1-z_4)}\in\mathbf C\cup\{\infty\}.
$$

\begin{Proposition}
\label{prop:killing1P}
Let $\mathfrak{a}, \mathfrak{b}\in\mathfrak{sl}_2(\mathbf C)$ be two nonzero infinitesimal isometries. Assume that $\mathfrak a$
is parabolic and $\mathfrak{b}$ is not.
Then
\begin{enumerate}
 \item $B(\mathfrak{a},\mathfrak{b})=0$ iff $\Fix(\mathfrak{a})$ is an endpoint of $\Axis(\mathfrak{b})$.
\item Let $p_+,p_-\in\mathbf C\cup\{\infty\}$ be the endpoints of $\Axis(\mathfrak{b})$, and assume that they are 
both different from 
 $\Fix(\mathfrak{a})$. Then 
$$
\frac{B(\mathfrak{a},\mathfrak{b})^2}{B(\mathfrak{b},\mathfrak{b})}= \frac{8}{t^2}[ p_+: e^{t\mathfrak{a}}(p_-): e^{t\mathfrak{a}}(p_+):p_-].
$$
\end{enumerate}
\end{Proposition}

\begin{proof}
Up to conjugacy, we may  assume that $\Fix(\mathfrak{a})=\infty$ in the upper half space model.
Then 
$$
\mathfrak{a}=\begin{pmatrix}
              0 & a \\ 0 & 0
             \end{pmatrix}
\textrm{ and }
\mathfrak{b}=\begin{pmatrix}
              u & v \\ w & -u
             \end{pmatrix}.
$$
With this expressions,  $B(\mathfrak{a},\mathfrak{b})= 4 a w$, and the first assertion of the proposition
follows from the fact that $w=0$ iff $\infty$ is an endpoint of the  axis of $\mathfrak{b}$.
To prove Assertion 2, we may assume up to further conjugation that the axis of $\mathfrak{b}$
has endpoints $ p_{\pm}=\pm x\in \mathbf  R\setminus \{0\}$. Hence 
$$
\mathfrak{a}=\begin{pmatrix}
              0 & a \\ 0 & 0
             \end{pmatrix}
\textrm{ and }
\mathfrak{b}=  \begin{pmatrix}
              0 &  y x \\ y/x  & 0
             \end{pmatrix}
$$
with $B(\mathfrak{b},\mathfrak{b})= 8 y^2$.
Thus:
$$
\frac{B(\mathfrak{a},\mathfrak{b})^2}{B(\mathfrak{b},\mathfrak{b})}
=
\frac{(4ay/x)^2}{8 y^2}= 2\frac{a^2}{x^2}.
$$
On the other hand,
$$
[ p_+: e^{t\mathfrak{a}}(p_-): e^{t\mathfrak{a}}(p_+):p_-]=
[x: -x+t\,a:x+t\, a: -x]= \frac{t^2 a^2}{4x^2},
$$
and the formula follows.
\end{proof}

Using  Proposition~\ref{prop:killing1P} and the computations in its proof we have the following remark:

\begin{Remark}
\label{rem:equidistant}
Assume $\mathfrak{a},\mathfrak{b}\in \mathfrak{sl}_2(\mathbf C)$ satisfy $B(\mathfrak{a},\mathfrak{b})\neq 0$,
$\mathfrak{a}$ is parabolic, and $\mathfrak{b}$ is not.
The horosphere centered at 
$\Fix(\mathfrak {a})\in\partial_{\infty}\mathbf H^3$ and tangent to
the axis $\Axis(\mathfrak{b})$ has a natural complex structure, up to homothety. 
Fix a complex structure in which $1\in\mathbf C$
is the unit tangent vector to $\Axis(\mathfrak{b})$, and 
suppose that  $tz\in\mathbf C$ is the translation vector of $e^{t\mathfrak{a}}$ in this horosphere. Then: 
$$
\frac{B(\mathfrak{a},\mathfrak{b})^2}{B(\mathfrak{b},\mathfrak{b})}= 2 z^2.
$$
\end{Remark}

By looking at the homothety factor of the complex structure on different horospheres with the same center, we get:

\begin{Corollary}
\label{cor:equihorsphere}
 Let $\mathfrak{a}, \mathfrak{b}, \mathfrak{c}\in\mathfrak{sl}_2(\mathbf C)$ be two nonzero infinitesimal isometries. Assume that $\mathfrak a$
is parabolic, but $\mathfrak{b}$ and $ \mathfrak{c}$ are not.
If
$$
\frac{B(\mathfrak{a},\mathfrak{b})^2}{B(\mathfrak{b},\mathfrak{b})}= 
\frac{B(\mathfrak{a},\mathfrak{c})^2}{B(\mathfrak{c},\mathfrak{c})}\neq 0 ,
$$
then the axis $\Axis(\mathfrak{b})$ and $\Axis(\mathfrak{c})$ are tangent to the same horosphere centered at $\Fix(\mathfrak {a})\in\partial_{\infty}\mathbf H^3$.
Moreover, their tangent directions are parallel in the Euclidean structure of the horosphere.
\end{Corollary}

Finally we deal with the case where $\mathfrak a$ and $\mathfrak b$ are both parabolic.

\begin{Proposition}
\label{prop:killing2P}
Let $\mathfrak{a}, \mathfrak{b}\in\mathfrak{sl}_2(\mathbf C)$ be two nonzero infinitesimal  parabolic isometries, with respective fixed points at infinity $\Fix(\mathfrak{a})$ and $\Fix(\mathfrak{b})$.
Then:
\begin{enumerate}
 \item $B(\mathfrak{a},\mathfrak{b})=0$ iff $\Fix(\mathfrak{a})=\Fix(\mathfrak{b})$.
\item 

When $\Fix(\mathfrak{a})\neq\Fix(\mathfrak{b})$, 
$$
B(\mathfrak{a},\mathfrak{b})= \frac{4}{t^2} [\Fix(\mathfrak{a}):\Fix(\mathfrak{b}):  
 e^{t\mathfrak{a}}(\Fix(\mathfrak{b})):
e^{t\mathfrak{b}}(\Fix(\mathfrak{a})) ] .
$$
\end{enumerate}
\end{Proposition}

\begin{proof}
The first assertion is an elementary computation, with a proof analogous to the first statement of Proposition~\ref{prop:killing1P}. For the second one, assume
$\Fix(\mathfrak{a})=\infty$ and $\Fix(\mathfrak{b})=0$. Hence 
$$
\mathfrak{a}=\begin{pmatrix}
              0 & x \\ 0 & 0
             \end{pmatrix}
\textrm{ and }
\mathfrak{b}=  \begin{pmatrix}
              0 & 0 \\ y  & 0
             \end{pmatrix}.
$$
Then $B(\mathfrak{a},\mathfrak{b})= 4 x y$. On the other hand, $e^{t\mathfrak{a}}(0)= t x$ and 
$e^{t\mathfrak{b}}(\infty) =1/(t y)$, and the formula is straightforward.
\end{proof}

\begin{footnotesize}

\noindent \textsc{Departament de Matem\`atiques, Universitat Aut\`onoma de Barcelona,\\  08193 Cerdanyola del Vall\`es (Spain)}

\noindent \emph{porti@mat.uab.cat}, \emph{porti@mat.uab.es}
\end{footnotesize}

\end{document}